\documentclass[a4paper,11pt]{amsart}
\usepackage{amsmath, amscd, amssymb, mathrsfs, mathtools, fullpage, enumerate, svninfo}
\usepackage[all]{xy}

\title{On $p$-adic interpolation of motivic Eisenstein classes}

\author{Guido Kings}
\address{Fakult\"at f\"ur Mathematik \\
Universit\"at Regensburg\\
93040 Regensburg\\
Germany}

\thanks{This research was supported by the DFG grant: SFB 1085 ``Higher invariants''}

\theoremstyle{plain}
    \newtheorem{theorem}{Theorem}[subsection]
\newtheorem*{theorem*}{Main Theorem}
    \newtheorem{lemma}[theorem]{Lemma}
    \newtheorem{proposition}[theorem]{Proposition}
    \newtheorem{corollary}[theorem]{Corollary}
    \newtheorem{definition}[theorem]{Definition}

\theoremstyle{remark}
    \newtheorem{remark}[theorem]{Remark}

\DeclareMathOperator{\TSym}{TSym}
\DeclareMathOperator{\Sym}{Sym}
\DeclareMathOperator{\Hom}{Hom}
\DeclareMathOperator{\Ext}{Ext}
\DeclareMathOperator{\Spec}{Spec}

\DeclareMathOperator{\pr}{pr}
\DeclareMathOperator{\Tr}{Tr}

\DeclareMathOperator{\mom}{mom}

\DeclareMathOperator{\Map}{Map}
\DeclareMathOperator{\rk}{rk}
\DeclareMathOperator{\EI}{\mathcal{EI}}

\newcommand{\one}[1]{\mathbf{1}^{(#1)}}
\newcommand{\bfone}{\mathbf{1}}

\newcommand{\et}{\text{\textup{\'et}}}
\newcommand{\mot}{\mathrm{mot}}

\newcommand{\res}{\mathrm{res}}

\newcommand{\sF}{\mathscr{F}}
\newcommand{\sG}{\mathscr{G}}
\newcommand{\sH}{\mathscr{H}}

\newcommand{\sS}{\mathscr{S}}

\newcommand{\sY}{\mathscr{Y}}
\newcommand{\sX}{\mathscr{X}}

\newcommand{\cA}{\mathcal{A}}
\newcommand{\cB}{\mathcal{B}}

\newcommand{\cL}{\mathcal{L}}

\newcommand{\QQ}{\mathbb{Q}}

\newcommand{\ZZ}{\mathbb{Z}}
\newcommand{\DD}{\mathbf{D}}

\newcommand{\Qp}{{\QQ_p}}
\newcommand{\Zp}{{\ZZ_p}}

\newcommand{\Eis}{\mathrm{Eis}}

\newcommand{\prolim}{\varprojlim}

\newcommand{\isom}{\cong}

\newcommand{\id}{\mathrm{id}}

\newcommand{\Log}[1]{\mathcal{L}og^{(#1)}}

\newcommand{\cpol}{\mathrm{pol}}

\newcommand{\cLog}{\mathcal{L}og}

\newcommand{\comp}{\operatorname{comp}}

\numberwithin{equation}{subsection}
\begin{document}

\begin{abstract}
In this paper we prove that the motivic Eisenstein classes associated to polylogarithms of commutative group schemes can be  $p$-adically  interpolated in \'etale cohomology. This connects them to Iwasawa theory and generalizes and strengthens the results for  elliptic curves obtained in our former work. In particular, degeneration questions can be treated  easily. 
\end{abstract}

\maketitle
\begin{center}
To John Coates, on the occasion of his 70th birthday
\end{center}
\tableofcontents
\section*{Introduction}
In this paper we prove that the motivic Eisenstein classes associated to polylogarithms of commutative group schemes can be  $p$-adically  interpolated in \'etale cohomology. This generalizes the results for  elliptic curves obtained in our former paper \cite{Kings-Soule}. Already in the one dimensional case the results obtained here are stronger and much more flexible as they  allow to  treat degenerating elliptic curves easily. 

The interpolation of  motivic Eisenstein classes connects them with Iwasawa theory and is essential for many applications. In the elliptic case for example, the interpolation was used in  \cite{Kings-Tamagawa} to prove a case of the Tamagawa number conjecture for CM elliptic curves and it was one of the essential ingredients in the proof of an explicit reciprocity law for Rankin-convolutions in \cite{KLZ}.
We hope that the general case will find similar applications.

Before we explain our results, we have to introduce
the motivic Eisenstein classes (for the construction we refer to Section \ref{sec:Qp-Eisenstein-classes}).

Let $\pi:G\to S$ be a smooth commutative and connected group scheme of relative dimension $d$ (for example a semi-abelian scheme) and  denote by  
\[
\sH:=R^{1}\pi_!\Zp(1)
\]
the first \'etale homology of $G/S$, which is just the sheaf of relative $p$-adic Tate modules of $G/S$. We write  $\sH_\Qp$ for the
associated $\Qp$-adic sheaf. Note that this is not a lisse sheaf in general. Evaluating the motivic polylogarithm 
at  a non-zero $N$-torsion section $t:S\to G$ one defines motivic 
Eisenstein classes 
\[
_\alpha\Eis_\mot^{k}(t)\in H^{2d-1}_\mot(S, \Sym^{k}\sH_\QQ(d)),
\]
depending on some auxiliary data $\alpha$, 
where $\Sym^{k}\sH_\QQ(1)$ is the $k$-th symmetric tensor power of the
motivic sheaf $\sH_\QQ$ which underlies $\sH_\Qp$. 

In the case of an elliptic curve, the  de Rham realization of $_\alpha\Eis_\mot^{k}(t)$ is the cohomology class of a  holomorphic Eisenstein series, which justifies the name. These motivic Eisenstein classes in the elliptic case play a major role in Beilinson's proof of his conjectures on special values of $L$-functions for modular forms. 

In this paper we consider the \'etale regulator 
\[
r_\et: H^{2d-1}_\mot(S, \Sym^{k}\sH_\QQ(d))\to H^{2d-1}(S, \Sym^{k}\sH_\Qp(d))
\]
which gives rise to the \'etale Eisenstein classes 
\[
_\alpha\Eis_\Qp^{k}(t):=r_\et(\Eis_\mot^{k}(t))\in H^{2d-1}(S, \Sym^{k}\sH_\Qp(d)).
\]
In the elliptic case these classes were used by Kato in his seminal work to construct Euler systems for modular forms.

It is a natural question, whether these \'etale Eisenstein classes enjoy some $p$-adic interpolation properties, in a similar way as one 
can $p$-adically interpolate the holomorphic Eisenstein series. At first sight, this seems to be a completely unreasonable question, as for varying $k$ the different
motivic cohomology groups $H^{2d-1}_\mot(S, \Sym^{k}\sH_\QQ(1))$ are not related at all. Nevertheless, this question was answered affirmatively in the elliptic case in \cite{Kings-Soule} and in this paper we will generalize this result to commutative group schemes.

To explain our answer to this question we need the 
\emph{sheaf of Iwasawa-algebras} $\Lambda(\sH)$, which is defined as follows: One first defines  a 
sheaf of "group rings" $\ZZ/p^{r}\ZZ[\sH_r]$ on $S$, where $\sH_r$ is the \'etale sheaf associated to the $[p^{r}]$-torsion subgroup 
$G[p^{r}]$ or alternatively the first homology with $\ZZ/p^{r}\ZZ$-coefficients (see section \ref{sec:iwasawa} for more details). These
group rings form an inverse system for varying $r$ and hence define  a pro-sheaf 
\[
\Lambda(\sH):=(\ZZ/p^{r}\ZZ[\sH_r])_{r\ge 0}.
\]
Moreover, it is also possible to sheafify the classical moments of a measure to 
a morphism of pro-sheaves
\[
\mom^{k}:\Lambda(\sH)\to \Gamma_{k}(\sH),
\]
where $\Gamma_k(\sH)$ is the $k$-th graded piece of the divided power algebra $\Gamma_\Zp(\sH)$. Thus the sheaf $\Lambda(\sH)$  $p$-adically interpolates the $\Gamma_{k}(\sH)$. For the $\Qp$-sheaf $\sH_\Qp$ the
natural map $\Sym^{k}\sH_\Qp\to \Gamma_k(\sH_\Qp)$ is an isomorphism
and the moment map gives rise to morphisms
\[
\mom^{k}: H^{2d-1}(S,\Lambda(\sH)(d)) \to H^{2d-1}(S,\Sym^{k}\sH_\Qp(d)).
\]
To understand this better, it is instructive to consider the 
case of an abelian scheme $\pi:A\to S$ over a scheme $S$ which is of finite type over $\Spec\ZZ$ (see also Section \ref{sec:abelian-schemes}). Then
\[
H^{2d-1}(S,\Lambda(\sH)(d))=\prolim_r H^{2d-1}(A[p^{r}],\ZZ/p^{r}\ZZ(d))
\]
where the inverse limit is taken with respect to the trace maps 
along $A[p^{r}]\to A[p^{r-1}]$. The right hand side is obviously an
Iwasawa theoretic construction. In the one dimensional case $d=1$, the right hand side has an interpretation as an inverse limit
of units via Kummer theory. 

Our main result can now be formulated as follows:
\begin{theorem*}[see Theorem \ref{thm:interpolation}]
There exists  a cohomology class  
\[
_\alpha\EI(t)_N\in H^{2d-1}(S,\Lambda(\sH)(d)) 
\]
called the \emph{Eisenstein-Iwasawa class},
such that 
\[
\mom^{k}(_\alpha\EI(t)_N)={N^{k}}{_\alpha}\Eis^{k}_\Qp(t).
\]
\end{theorem*} 
This interpolation result in the elliptic case is one of the key ingredients in the proof of
an explicit reciprocity law for Rankin-convolutions of modular forms in
\cite{KLZ}. 

The use of this theorem also considerably simplifies the computations  of the degeneration of polylogarithm in \cite{Kings-Soule}. We hope to treat this at another occasion.

We would also like to point out an important open problem:
In the one-dimensional torus or the elliptic curve case, the Eisenstein-Iwasawa class has a direct description in terms of cyclotomic units or  Kato's norm compatible elliptic units respectively. 
Unfortunately, we do not have a similar description of the Eisenstein-Iwasawa class in the higher dimensional case. 
\section*{Acknowledgements}
This paper is a sequel to \cite{Kings-Soule}, which was written as a response to John Coates' wish to have an exposition of the results in \cite{HuKi} at a conference in Pune, India. This triggered a renewed interest of mine in the $p$-adic interpolation of motivic Eisenstein classes which goes back to \cite{Kings-Tamagawa}. In this sense the paper would not exist without the persistence of John Coates. 
The paper \cite{HuKiPol} with Annette Huber   
created the right framework to treat these questions. It is a pleasure  to thank them both and to dedicate this paper to John Coates on the occasion of his seventieth birthday.
\section{Notations and set up}

\subsection{The category of $\Zp$-sheaves}
\label{section-Zp-sheaves} All schemes will be separated of finite type over a noetherian regular scheme of dimension $0$ or $1$. 
Let $X$ be such a scheme  and let $p$ be a prime invertible on $X$. We work in the category of constructible $\Zp$-sheaves  $\sS(X)$ on $X$ in the
sense of \cite[Expos\'e V]{sga5}. 

Recall that a constructible $\Zp$-sheaf is 
an inverse system $\sF=(\sF_n)_{n\ge 1}$ where $\sF_n$ is a constructible $\ZZ/p^{n}\ZZ$-sheaf and the transition maps
$\sF_n\to \sF_{n-1}$ factor into isomorphisms
\[
\sF_n\otimes_{\ZZ/p^{n}\ZZ}\ZZ/p^{n-1}\ZZ\isom \sF_{n-1}.
\]
The $\Zp$-sheaf is \emph{lisse}, if each $\sF_n$ is locally constant. If $X$ is connected and $x\in X$ is a geometric point, then the category of lisse sheaves is equivalent to the category of finitely generated $\Zp$-modules with a continous $\pi_1(X,x)$-action. For a general $\Zp$-sheaf there exists a finite partition of $X$ into locally closed subschemes $X_i$, such that $\sF\mid_{X_i}$ is lisse (see \cite[Rapport, Prop. 2.4., 2.5.]{SGA41/2}).

For a $\Zp$-sheaf $\sF$ we denote by $\sF\otimes\Qp$ its image in the
category of $\Qp$-sheaves, i.e., the quotient category modulo $\Zp$-torsion sheaves.

We also consider the ``derived'' category $\DD(X)$ of $\sS(X)$
in the sense of Ekedahl \cite{Eke}. This is a triangulated category with a $t$-structure whose heart is the category of constructible $\Zp$-sheaves.
By loc. cit. Theorem 6.3 there is a full 6 functor formalism on these categories. 

Recall that an inverse system $A:=(A_r)_{r\ge 0}$ (in any abelian category $\cA$) satisfies the Mittag-Leffler condition (resp. is Mittag-Leffler zero), if for each $r$ the
image of $A_{r+s}\to A_r$ is constant for all sufficiently big $s$ (is zero for some $s\ge 1$).
If $A$ satisfies the 
Mittag-Leffler condition and $\cA$ satisfies $AB4^{*}$ (i.e. products exists and products of epimorphisms are epimorphisms)  then 
${\prolim}_r^{1}A_r=0$ ( see \cite[Proposition 1]{Roos}).
If $A$ is Mittag-Leffler zero, then for 
each left exact functor $h:\cA\to \cB$ one has $R^{i}\prolim_r h(A_r)=0$ for all $i\ge 0$ (\cite[Lemma 1.11.]{Jannsen-cont}).

For a pro-system of \'etale sheaves $\sF=(\sF_r)_{r\ge 0}$ on $X$ we work with Jannsen's continuous \'etale cohomology $H^{i}(X,\sF)$ which is the $i$-th derived functor of $\sF\mapsto \prolim_r H^{0}(X,\sF_n)$.
By \cite[3.1]{Jannsen-cont} one has an exact sequence
\begin{equation}
0\to {\prolim}_r^{1}H^{i-1}(X,\sF_r)\to H^{i}(X,\sF)\to \prolim_rH^{i}(X,\sF_r)\to 0.
\end{equation}
Note in particular, that if $H^{i-1}(X,\sF_r)$ is finite for all $r$, one has
\begin{equation}\label{eq:cont-coh}
H^{i}(X,\sF)=\prolim_r H^{i}(X,\sF_r).
\end{equation}
For $\sF=(\sF_r)$ Mittag-Leffler zero, one has for all $i\ge 0$
\begin{equation}\label{eq:vanishing-MLzero}
H^{i}(X,\sF)=0.
\end{equation}

\subsection{The divided power algebra} Let $A$ be a commutative ring and $M$ be an $A$-module. Besides the usual symmetric power algebra $\Sym_A(M)$ we need also the  divided power algebra $\Gamma_A(M)$ (see \cite[Appendix A]{BeOg} for more details).

The $A$-algebra $\Gamma_A(M)$ is a graded augmented algebra with
$\Gamma_0(M)=A$, $\Gamma_1(M)=M$ and augmentation ideal 
$\Gamma^{+}(M):=\bigoplus_{k\ge 1}\Gamma_k(M)$. For each element $m\in M$ one has the divided power $m^{[k]}\in \Gamma_k(M)$ with the property that $m^{k}=k! m^{[k]}$ where 
$m^{k}$ denotes the $k$-th power of $m$ in $\Gamma_A(M)$. 
Moreover, one has the formula
\[
(m+n)^{[k]}=\sum_{i+j=k}m^{[i]}n^{[j]}.
\]
In the case where $M$ is a free $A$-module with basis $m_1,\ldots,m_r$ the $A$-module $\Gamma_k(M)$ is free with 
basis $\{m_1^{[i_1]}\cdots m_r^{[i_r]}\mid \sum i_j=k\}$. Further, for $M$ free, there is an $A$-algebra isomorphism
\[
\Gamma_A(M)\isom \TSym_A(M)
\]
with the algebra of symmetric tensors ($\TSym^{k}_A(M)\subset \Sym^{k}_A(M)$ are the invariants of the symmetric group), which maps $m^{[k]}$ to 
$m^{\otimes k}$. Also, by the universal property of $\Sym_A(M)$, one has an $A$-algebra homomorphism
\begin{equation}\label{eq:Sym-und-Gamma}
\Sym_A(M)\to \Gamma_A(M)
\end{equation}
which maps $m^{k}$ to $k!m^{[k]}$. In particular, if $A$ is a $\QQ$-algebra, this  map is an isomorphism.

If $M$ is free and $M^{\vee}:=\Hom_A(M,A)$ denotes the $A$-dual one has in particular
\[
\Sym^{k}(M^{\vee})\isom\Gamma_k(M)^{\vee}\isom \TSym^{k}_A(M)^{\vee}.
\]
The algebra  $\Gamma_A(M)$ has the advantage over $\TSym_A(M)$ of being compatible with arbitrary base change 
\[
\Gamma_A(M)\otimes_AB\isom \Gamma_B(M\otimes_AB)
\]
and thus sheafifies well. 
Recall from
\cite[I 4.2.2.6.]{Illusie} that if $\sF$ is an \'etale sheaf of $\Zp$-modules, then  $\Gamma_{\Zp}(\sF)$ is defined
to be the sheaf associated to the presheaf
\begin{equation}
U\mapsto \Gamma_{\Zp(U)}(\sF(U)).
\end{equation}
\begin{definition}
We denote by 
\[
\widehat\Gamma_A(M):=\prolim_r \Gamma_A(M)/\Gamma^{+}(M)^{[r]}
\]
the completion of $\Gamma_A(M)$ with respect to the divided powers of the 
augmentation ideal. 
\end{definition}
Note that $\Gamma^{+}(M)^{[r]}=\bigoplus_{k\ge r}\Gamma_k(M)$ so that as $A$-module one has $\widehat\Gamma_A(M)\isom \prod_{k\ge 0}\Gamma_k(M)$.

In the same way we define the completion of $\Sym_A(M)$ with respect
to the augmentation ideal $\Sym^{+}_A(M)$ to be
\begin{equation}
\widehat{\Sym}_A(M):=\prolim_k \Sym_A(M)/(\Sym^{+}_A(M))^{k}
\end{equation}
\subsection{Unipotent sheaves}\label{sec:unipotent}Let $\Lambda=\ZZ/p^{r}\ZZ, \Zp$ or $\Qp$ and let $\pi:X\to S$ be a separated scheme of finite type, with $X,S$ as in \ref{section-Zp-sheaves}. 
A $\Lambda$-sheaf $\sF$ on $X$ is \emph{unipotent of length $n$}, if it has a filtration
$0=\sF^{n+1}\subset \sF^n\subset \ldots\subset \sF^0=\sF$ such that 
$\sF^i/\sF^{i+1}\isom\pi^*\sG^i$ for a $\Lambda$-sheaf $\sG^i$ on $S$.

The next lemma is taken from \cite{HuKiPol}, where it is stated in 
the setting of $\Qp$-sheaves. 
\begin{lemma}\label{lem:unipotent_uppershriek}Let $\Lambda=\ZZ/p^{r}\ZZ, \Zp$ or $\Qp$ and
let $\pi_1:X_1\to S$ and $\pi_2:X_2\to S$ be smooth of constant fibre dimension $d_1$ and $d_2$. Let $f:X_1\to X_2$ be an $S$-morphism.
Let $\sF$ be a unipotent $\Lambda$-sheaf. Then
\[ f^!\sF=f^*\sF(d_1-d_2)[2d_1-2d_2].\]
\end{lemma}
\begin{proof} Put $c=d_1-d_2$ the relative dimension of $f$.
We start with the case $\sF=\pi_2^*\sG$. 
In this case
\begin{multline*}
f^!\sF=f^!\pi_2^*\sG=f^!\pi_2^!\sG(-d_2)[-2d_2]=\pi_1^!\sG(-d_2)[-2d_2]\\
=\pi_1^*\sG(c)[2c]=f^*\pi_2^*\sG(c)[2c]
=f^*\sF\otimes\Lambda(c)[2c].
\end{multline*}
In particular, $f^!\Lambda=\Lambda(c)[2c]$ and we may rewrite the 
 formula as
\[ f^*\sF\otimes  f^!\Lambda=f^!(\sF\otimes \Lambda).\]
There is always a map from the left to right via adjunction from
the projection formula
\[ Rf_!(f^*\sF\otimes f^!\Lambda)= \sF\otimes Rf_!f^!\Lambda\to \sF\otimes \Lambda.\]
Hence we can argue on the unipotent length of $\sF$ and it suffices
to consider the case $\sF=\pi^*\sG$. This case was settled above. \end{proof}
The next lemma is also taken from \cite{HuKiPol}.
Let $X\to S$ be a smooth scheme with connected fibres and $e:S\to X$ a section. Homomorphisms of unipotent sheaves are completely determined by their restriction to $S$ via $e^{*}$:
\begin{lemma} \label{lemma:e-upper-star}
Let $\pi:X\to S$ be smooth with connected fibres and $e:S\to X$ a section of $\pi$. Let $\Lambda=\ZZ/p^{r}\ZZ, \Zp$ or $\Qp$ and $\sF$ a unipotent $\Lambda$-sheaf on $X$. Then
\[
e^{*}:\Hom_X(\Lambda,\sF)\to \Hom_S(\Lambda,e^{*}\sF)
\]
is injective. 
\end{lemma}
\begin{proof}
Let $0\to\sF_1\to \sF_2\to \sF_3\to 0$ be a short exact sequence of 
unipotent $\Lambda$-sheaves on $G$. 
As  $e^*$ is exact  and  $\Hom$ left exact, we get a commutative
diagram of exact sequences
\[\xymatrix{%
0\ar[r]&\Hom_X(\Lambda,\sF_1)\ar[r]\ar[d]&\Hom_X(\Lambda,\sF_2)\ar[r]\ar[d]&\Hom_X(\Lambda,\sF_3)\ar[d]\\
0\ar[r]&\Hom_S(\Lambda,e^*\sF_1)\ar[r]&\Hom_S(\Lambda,e^*\sF_2)\ar[r]&\Hom_S(\Lambda,e^*\sF_3).
}\]
Suppose that the left and right vertical arrows are injective, then the middle one is injective as well and it is enough to show the lemma
in the case where $\sF=\pi^{*}\sG$.
But the isomorphism
\[
\Hom_X(\Lambda,\pi^{*}\sG)\isom \Hom_X(\pi^{!}\Lambda,\pi^{!}\sG)\isom
\Hom_S(R\pi_!\pi^{!}\Lambda,\sG)
\]
factors through 
\[
\Hom_X(\pi^{!}\Lambda,\pi^{!}\sG)\xrightarrow{e^{!}}\Hom_S(\Lambda,\sG)\to
\Hom_S(R\pi_!\pi^{!}\Lambda,\sG)
\]
where the last map is induced by the trace map $R\pi_!\pi^{!}\Lambda\to \Lambda$. This proves the claim.
\end{proof}

\subsection{The geometric situation}\label{sec:geom-sit}
We recall the geometric set up from \cite{HuKiPol} using as much as possible the notations from loc. cit.
Let \[ \pi:G\to S\]
be a smooth separated commutative group scheme with connected fibres of relative dimension $d$. We denote by $e:S\to G$ the unit section and by $ \mu:G\times_SG\to G$ the multiplication .
Let $j:U\to G$ be the open complement of $e(S)$.

Let $\iota_D:D\to G$ be a closed subscheme with structural map
$\pi_D:D\to S$. Typically $\pi_D$ will be \'etale and
contained in the $c$-torsion of $G$ for some $c\geq 1$.  We note in passing, that for $c$ invertible on $S$ the $c$-torsion points of $G$, i.e. the kernel of the $c$-multiplication $G[c]$, is quasi-finite and \'etale over $S$.  
Denote by
$j_D:U_D=G\setminus D\to G$ the open complement of $D$. We summarize the situation in the basic diagram
\[
\xymatrix{
U_D:=G\setminus D\ar[r]^/.6em/{j_D}\ar[rd]&G\ar[d]^{\pi}&D\ar[l]_{\iota_D}\ar[dl]^{\pi_D}\\
&S
}
\]
We will also consider morphisms $\phi:G_1\to G_2$ of $S$-group schemes as above. In this case we decorate all notation with an index $1$ or $2$, e.g., $d_1$
for the relative dimension of $G_1/S$.

\section{The logarithm sheaf}
\subsection{Homology of $G$}
The basic sheaf in our constructions is the relative first $\Zp$-homology $\sH_G$ of $G/S$, which we define as follows:
\begin{definition}\label{defn:tate}
For the group scheme $\pi:G\to S$ we let 
\[
\sH:=\sH_G:=R^{2d-1}\pi_!\Zp(d)=R^{-1}\pi_!\pi^!\Zp.
\]
We write $\sH_r:=\sH\otimes\ZZ/p^{r}\ZZ$ and $\sH_{\Qp}:=\sH\otimes \Qp$ for the associated $\Qp$-sheaf.
\end{definition}
Note that $\sH$ is not a lisse $\Zp$-sheaf in general, but the stalks
are free $\Zp$-modules of finite rank, which follows for example from
Lemma \ref{lemma:description-of-H} below.

The sheaf $\sH$ and more generally $R^{i}\pi_!\Zp$ is covariant functorial for any map of $S$-schemes $f:G\to X$ using the adjunction $f_!f^{!}\Zp\to \Zp$. In particular,
the group multiplication $\mu:G\times_SG \to G$ induces a product
\[
R^{i}\pi_!\Zp(d)\otimes R^{j}\pi_!\Zp(d)\to R^{i+j-2d}\pi_!\Zp(d)
\]
and the diagonal $\Delta:G\to G\times_SG$ induces a coproduct 
\[
R^{i}\pi_!\Zp(d)\to \bigoplus_{j}R^{j}\pi_!\Zp(d)\otimes R^{2d+i-j}\pi_!\Zp(d)
\]
on $R^{\cdot}\pi_!\Zp$, which gives it the structure of a Hopf algebra 
and one has
\begin{equation}\label{eq:Hopf-alg}
R^{i}\pi_!\Zp(d)\isom \bigwedge^{2d-i}\sH
\end{equation}
(this follows by base change to geometric points and duality from \cite[Lemma 4.1.]{BrSz}). The same result holds for $\ZZ/p^{r}\ZZ$-coefficients.
\begin{lemma}\label{lemma:description-of-H}
Let $G[p^{r}]$ be the kernel of the $p^{r}$-multiplication $[p^{r}]:G\to G$. Then there is a canonical isomorphism of \'etale sheaves
\[
G[p^{r}]\isom R^{-1}\pi_!\pi^!\ZZ/p^{r}\ZZ=\sH_r.
\]
In particular, $\sH_G$ is the $p$-adic Tate-module of $G$.
\end{lemma}
\begin{proof}
This is standard and we only sketch the proof: Consider 
$G[p^{r}]$ as an \'etale sheaf on $S$. The Kummer sequence
is a $G[p^{r}]$-torsor on $G$, hence gives a class in 
\begin{multline*}
H^{1}(G,\pi^{*}G[p^{r}])\isom \Ext^{1}_G(\pi^{*}\ZZ/p^{r}\ZZ,\pi^{*}G[p^{r}])\isom \Ext^{1}_G(\pi^{!}\ZZ/p^{r}\ZZ,\pi^{!}G[p^{r}])\isom \\
\isom \Ext^{1}_S(R\pi_!\pi^{!}\ZZ/p^{r}\ZZ,G[p^{r}])\isom \Hom_S(R^{-1}\pi_!\pi^!\ZZ/p^{r}\ZZ,G[p^{r}]).
\end{multline*}
Thus the Kummer torsor induces a map $R^{-1}\pi_!\pi^!\ZZ/p^{r}\ZZ\to G[p^{r}]$ and one can perform a base change to geometric points $\overline{s}\in S$ 
to show that this is an isomorphism. But this follows then from Poincar\'e-duality and the isomorphism $\Hom_{\overline{s}}(G[p^{r}],\mu_{p^{r}})\isom H^{1}(G,\mu_{p^{r}})$
shown in \cite[Lemma 4.2.]{BrSz}.
\end{proof}

\subsection{The first logarithm sheaf}
Consider the complex $R\pi_!\pi^!\Zp$  calculating the homology of $\pi:G\to S$ and its canonical filtration 
whose associated graded pieces are the $R^{i}\pi_!\pi^!\Zp$.
We apply this to 
\[
R\Hom_G(\pi^!\Zp,\pi^!\sH)\isom R\Hom_S(R\pi_!\pi^{!}\Zp,\sH).
\]
Then the resulting hypercohomology spectral sequence 
gives rise to the five term sequence
\begin{multline*}
0\to \Ext^1_S(\Zp,\sH)\xrightarrow{\pi^!}
\Ext^1_G(\pi^!\Zp,\pi^!\sH)\to 
\Hom_S(\sH,\sH)\to\\
\to
\Ext^2_S(\Zp,\sH)\xrightarrow{\pi^!} \Ext^2_G(\pi^!\Zp,\pi^!\sH)
\end{multline*}
and the maps $\pi^{!}$ are injective because they admit
the splitting $e^{!}$ induced by the unit section $e$. This gives
\begin{equation}\label{eq:Log-1}
0\to \Ext^1_S(\Zp,\sH)\xrightarrow{\pi^!}
\Ext^1_G(\pi^!\Zp,\pi^!\sH)\to 
\Hom_S(\sH,\sH)\to 0.
\end{equation}
Note that $\Ext^1_G(\pi^!\Zp,\pi^!\sH)\isom \Ext^1_G(\Zp,\pi^*\sH)$. The same construction is also possible with the base ring $\Lambda_r:=\ZZ/p^{r}\ZZ$ and
$\sH_r$ and gives the exact sequence
\begin{equation}\label{eq:Log-Lambda}
0\to \Ext^1_S(\Lambda_r,\sH_r)\xrightarrow{\pi^!}
\Ext^1_G(\pi^!\Lambda_r,\pi^!\sH_r)\to 
\Hom_S(\sH_r,\sH_r)\to 0.
\end{equation}
\begin{definition}\label{defn:log}
The \emph{first logarithm sheaf} $(\Log{1},\one{1})$ on $G$ consists of an extension class
\[
0\to \pi^*\sH\to \Log{1}\to \Zp\to 0
\]
such that its image in $\Hom_S(\sH,\sH)$ is the identity
together with
a fixed splitting $\mathbf{1}^{(1)}:e^*\Zp\to e^*\Log{1}$. In exactly the same way one defines $\Log{1}_{\Lambda_r}$. 
We denote by $\Log{1}_{\Qp}$ the associated $\Qp$-sheaf.
\end{definition}

The existence and uniqueness of $(\Log{1},\mathbf{1}^{(1)})$ follow directly from \eqref{eq:Log-1}. The automorphisms of $\Log{1}$ form a torsor under $\Hom_G(\Zp,\pi^{*}\sH)$. In particular, the 
pair $(\Log{1},\mathbf{1}^{(1)})$ admits no automorphisms except the
identity.

It is obvious from the definition that one has
\begin{equation}
\Log{1}\otimes_\Zp\Lambda_r\isom \Log{1}_{\Lambda_r}
\end{equation}
so that $\Log{1}=(\Log{1}_{\Lambda_r})_{r\ge 0}$. Moreover, 
$\Log{1}$ is compatible with arbitrary base change. If 
\begin{equation}
\begin{CD}
G_T@>f_T>> G\\
@V\pi_T VV@VV\pi V\\
T@>f>> S
\end{CD}
\end{equation}
is a cartesian diagram one has $f_T^{*}\Log{1}_G\isom \Log{1}_{G_T}$
and $f_T^{*}(\one{1})$ defines a splitting.

Let 
\[
\varphi:G_1\to G_2
\] 
be a homomorphism of group schemes of relative 
dimension $d_1$, $d_2$, respectively and write $c:=d_1-d_2$.
\begin{theorem}\label{thm:log-1-functoriality}
For  $\varphi:G_1\to G_2$ as above, there is a unique morphism of sheaves
\[
\varphi_\#:\Log{1}_{G_1}\to\varphi^{*}\Log{1}_{G_2}\isom \varphi^{!}\Log{1}_{G_2}(-c)[-2c]
\]
such that $\varphi_\#(\one{1}_{G_1})=\one{1}_{G_2}$. Moreover, if $\varphi$ is an isogeny of
degree prime to $p$, then $\varphi_\#$ is an isomorphism.
\end{theorem}
\begin{proof}
Pull-back of $\Log{1}_{G_2}$ gives an exact sequence
\[
0\to \pi_1^{*}\sH_{G_2}\to \varphi^{*}\Log{1}_{G_2}\to \Zp\to 0
\]
and push-out of $\Log{1}_{G_1}$ by $\pi_1^{*}\sH_{G_1}\to\pi_1^{*}\sH_{G_2}$ induces a map 
\begin{equation}
\begin{CD}
0@>>> \pi_1^{*}\sH_{G_1}@>>>\Log{1}_{G_1}@>>> \Zp@>>> 0\\
@.@VVV@VV h V@|\\
0@>>> \pi_1^{*}\sH_{G_2}@>>>\varphi^{*}\Log{1}_{G_2}@>>> \Zp@>>> 0.
\end{CD}
\end{equation}
If $\varphi$ is an isogeny and $\deg\varphi$ is prime to $p$, then $\pi_1^{*}\sH_{G_1}\to\pi_1^{*}\sH_{G_2}$ is an isomorphism, hence also $h$. By uniqueness there is a unique isomorphism of the pair $(\Log{1}_{G_2},e_1^{*}(h)\circ \one{1}_{G_1})$ with $(\Log{1}_{G_2},\one{1}_{G_2})$. The composition of this isomorphism with $h$ is the desired map. If $h':\Log{1}_{G_1}\to \varphi^{!}\Log{1}_{G_2}$ is another map with this property, the difference $h-h':\Zp\to \pi_1^{*}\sH_{G_2}$ is uniquely determined by 
its pull-back $e^{*}(h-h'):\Zp\to e_2^{*}\Log{1}_{G_2}$ according to 
Lemma \ref{lemma:e-upper-star}. If both, $h$ and $h'$ are compatible with the splittings, then $e^{*}(h-h')=0$ and hence $h=h'$.
\end{proof}
\begin{corollary}[Splitting principle]\label{cor:log-1-splitting}
Let $\varphi:G_1\to G_2$ be an isogeny of degree prime to $p$. Then if 
$t:S\to G_1$ is in the kernel of $\varphi$, then
\[
t^*\Log{1}_{G_1}\isom t^{*}\varphi^{*}\Log{1}_{G_2}\isom 
e_1^{*}\varphi^{*}\Log{1}_{G_2}\isom e_1^{*}\Log{1}_{G_1}.
\]
\end{corollary}
\begin{proof}
Apply $t^{*}$ to $\varphi_\#$.
\end{proof}
\subsection{The $\Qp$-logarithm sheaf}\label{sec:Qp-logarithm}
We are going to define the $\Qp$-logarithm sheaf, which has been studied extensively in \cite{HuKiPol}.
\begin{definition}
We define
\[
\Log{k}_\Qp:=\Sym^{k}(\Log{1}_\Qp)
\]
and denote by 
\[
\one{k}:=\frac{1}{k!}\Sym^{k}(\one{1}):\Qp\to \Log{k}_\Qp
\]
the splitting induced by $\one{1}$.
\end{definition}
We note that $\Log{k}_\Qp$ is unipotent of length $k$ and 
that the splitting $\one{k}$ induces an isomorphism
\begin{equation}\label{eq:log-splitting}
e^{*}\Log{k}_\Qp\isom \prod_{i=0}^{k}\Sym^{i}\sH_\Qp.
\end{equation} 
To define 
transition maps 
\begin{equation}
\Log{k}_\Qp\to \Log{k-1}_\Qp
\end{equation}
consider the morphism $\Log{1}_\Qp\to \Qp\oplus\Log{1}_\Qp$ given by the canonical projection and the identity. Then we have
\begin{multline*}
\Log{k}_\Qp=\Sym^{k}(\Log{1}_\Qp)\to \Sym^{k}(\Qp\oplus\Log{1}_\Qp)\isom
\bigoplus_{i+j=k}\Sym^{i}(\Qp)\otimes\Sym^{j}(\Log{1}_\Qp)\to\\
\to
\Sym^{1}(\Qp)\otimes\Sym^{k-1}(\Log{1}_\Qp)\isom \Log{k-1}_\Qp.
\end{multline*}
A straightforward computation shows that $\one{k}\mapsto \one{k-1}$ under this transition map. 
\subsection{Main properties of the $\Qp$-logarithm sheaf}
The logarithm sheaf has three main properties: functoriality, 
vanishing of cohomology and  a universal mapping property for
unipotent sheaves. Functoriality follows trivially from Theorem \ref{thm:log-1-functoriality}. We review the others briefly, referring for more details to \cite{HuKiPol}.

Let $\varphi:G_1\to G_2$
be a homomorphism of group schemes of relative 
dimension $d_1$, $d_2$, respectively and let  $c:=d_1-d_2$ be the relative dimension of the homomorphism.
\begin{theorem}[Functoriality] \label{thm:functoriality} For 
 $\varphi:G_1\to G_2$ as above there is a unique
homomorphism of sheaves
\[
\varphi_\#:\cLog_{\Qp,G_1}\to \varphi^{*}\cLog_{\Qp,G_2}\isom \varphi^!\cLog_{\Qp,G_2}(-c)[-2c]
\]
such that $\mathbf{1}_{G_1}$ maps to $\mathbf{1}_{G_2}$. Moreover, if $\varphi$ is an isogeny, the $\varphi_\#$ is an 
isomorphism. 
\end{theorem}
\begin{proof}
This follows
directly from Theorem \ref{thm:log-1-functoriality} and the fact that
$\deg\varphi$ is invertible in $\Qp$.
\end{proof}
\begin{corollary}[Splitting principle]\label{cor:splitting_sheaf}
Let $\varphi:G_1\to G_2$ be an isogeny. Then if 
$t:S\to G_1$ is in the kernel of $\varphi$, one has
\[
\varrho_t:t^*\cLog_{\Qp,G_1}\isom t^{*}\varphi^{*}\cLog_{\Qp,G_2}\isom 
e_1^{*}\varphi^{*}\cLog_{\Qp,G_2}\isom e_1^{*}\cLog_{\Qp,G_1}\isom
\prod_{k\ge 0}\Sym^k\sH_{\Qp,G_1}.
\]
More generally, if $\iota:\ker\varphi\to G_1$ is the closed immersion, one has 
\[
\iota^{*}\cLog_{G_1}\isom \pi\mid_{\ker\varphi}^{*}\prod_{k\ge 0}\Sym^k\sH_{\Qp,G_1},
\]
where $\pi\mid_{\ker\varphi}:\ker\varphi\to S$ is the structure map.
\end{corollary}
\begin{proof}
Apply $t^{*}$ to both sides of the isomorphism $\varphi_{\#}$ and use \eqref{eq:log-splitting}. For the second statement make the base change
to $\ker\varphi$ and apply the first statement to the tautological section
of $\ker\varphi$.
\end{proof}
\begin{theorem}[Vanishing of cohomology]\label{thm:vanishing}
One has 
\[
R^i\pi_!\cLog_\Qp \isom \begin{cases}
\Qp(-d) &\mbox{if } i=2d\\
0 & \mbox{if }i\neq 2d.
\end{cases}
\]
More precisely, the transition maps $R^i\pi_!\Log{k}_\Qp \to R^i\pi_!\Log{k-1}_\Qp$ are zero for $i<2d$ and one has an
isomorphism $R^{2d}\pi_!\Log{k}_\Qp\isom \Qp(-d)$ compatible with the transition maps.  
\end{theorem}
\begin{proof}
This is Theorem 3.3.1. in \cite{HuKiPol}.
\end{proof}
Let $\sF$
be a unipotent sheaf of finite length $n$ on $G$. Consider the homomorphism
\begin{equation}\label{eq:universal-property-map}
\pi_*\underline{\Hom}_G(\cLog_\Qp,\sF)\to e^*\sF
\end{equation}
defined as the composition of 
\[
\pi_*\underline{\Hom}_G(\cLog_\Qp,\sF)\to \pi_*e_*e^{*}\underline{\Hom}_G(\cLog_\Qp,\sF)
\to \underline{\Hom}_S(e^*\cLog_\Qp,e^*\sF)
\] 
with
\[
\underline{\Hom}_S(e^*\cLog_\Qp,e^*\sF)\xrightarrow{(\mathbf{1})^{*}}
\underline{\Hom}_S(\Qp,e^*\sF)\isom e^*\sF.
\]
\begin{theorem}[Universal property]\label{thm:sheaf-univ-property}
Let $\sF$ be a unipotent sheaf of finite length. Then the  map \eqref{eq:universal-property-map} induces an isomorphism
\[
\pi_*\underline\Hom(\cLog_\Qp,\sF)\isom e^*\sF.
\]
\end{theorem}
\begin{proof}
This is Theorem 3.3.2. in \cite{HuKiPol}.
\end{proof}
\section{The $\Qp$-polylogarithm and Eisenstein classes}
\subsection{Construction of the $\Qp$-polylogarithm}\label{sec:Qp-polylog}
Fix an auxiliary  integer $c>1$ invertible on $S$ and consider the $c$-torsion subgroup $D:=G[c]\subset G$. We write $U_D:=G\setminus D$ and consider
\[
U_D\xrightarrow{j_D}G\xleftarrow{\iota_D}D.
\]
We also write $\pi_D:D\to S$ for the structure map. 

For any sheaf $\sF$ the localization triangle defines a connecting homomorphism
\begin{equation}\label{eq:localization-seq}
R\pi_!Rj_{D*}j_D^{*}\sF[-1]\to R\pi_!\iota_{D!}\iota_D^{!}\sF.
\end{equation}
As $\Log{k}_\Qp(d)[2d]$ is unipotent we may use  Lemma \ref{lem:unipotent_uppershriek} to replace 
$\iota_D^{!}$ by $\iota_D^{*}$. Using Corollary \ref{cor:splitting_sheaf} one gets
\[
\pi_{D!}\iota_D^{!}\Log{k}_\Qp(d)[2d]\isom \prod_{i=0}^{k}
\pi_{D!}\pi_D^{*}\Sym^{i}\sH_\Qp.
\]
Putting everything together and taking the limit over 
the transition maps $\Log{k}_\Qp\to \Log{k-1}_\Qp$ gives the \emph{residue  map}
\begin{equation}
\res: H^{2d-1}(S,R\pi_!Rj_{D*}j_D^{*}\cLog_\Qp(d))\to
H^{0}(S,\prod_{k\ge 0}
\pi_{D!}\pi_D^{*}\Sym^{k}\sH_\Qp).
\end{equation}
\begin{proposition}
The localization triangle induces  a short exact sequence
\[
0\to H^{2d-1}(S,R\pi_!Rj_{D*}j_D^{*}\cLog_\Qp(d))\xrightarrow{\res}
H^{0}(S,\prod_{k\ge 0}
\pi_{D!}\pi_D^{*}\Sym^{k}\sH_\Qp)\to 
H^{0}(S,\Qp)\to 0.
\]
\end{proposition}
\begin{proof}
This is an immediate consequence from the localization triangle and the computation of $R\pi_!\cLog$ in Theorem \ref{thm:vanishing}.
\end{proof}
\begin{definition}
Let
\[
\Qp[D]^{0}:=\ker(H^{0}(S,\pi_{D!}\Qp)\to H^{0}(S,\Qp))
\]
where the map is induced by the trace $\pi_{D!}\Qp\to \Qp$.
\end{definition}
Note that 
\[
\Qp[D]^{0}\subset \ker \left(H^{0}(S,\prod_{k\ge 0}
\pi_{D!}\pi_D^{*}\Sym^{k}\sH_\Qp)\to 
H^{0}(S,\Qp)
\right).
\]
\begin{definition}
Let $\alpha\in \Qp[D]^{0}$. Then the unique class
\[
_{\alpha}\cpol_\Qp\in H^{2d-1}(S,R\pi_!Rj_{D*}j_D^{*}\cLog_\Qp(d))
\]
with $\res({_\alpha}\cpol_\Qp)=\alpha$ is called the \emph{polylogarithm class associated to $\alpha$}. We write
$_{\alpha}\cpol_\Qp^{k}$ for the image of $_{\alpha}\cpol_\Qp$ in
$H^{2d-1}(S,R\pi_!Rj_{D*}j_D^{*}\Log{k}_\Qp(d))$.
\end{definition}

\subsection{Eisenstein classes}\label{sec:Qp-Eisenstein-classes}
Recall that $D=G[c]$ and fix an integer $N>1$ invertible on $S$, such that $(N,c)=1$ and let $t:S\to U_D=G\setminus D$ be an $N$-torsion section. Consider the composition
\begin{equation}\label{eq:t-evaluation}
R\pi_!Rj_{D*}j_D^{*}\cLog(d)\to R\pi_!Rj_{D*}j_D^{*}Rt_*t^{*}\cLog(d)
\isom R\pi_!Rt_*t^{*}\cLog(d)\isom t^{*}\cLog(d)
\end{equation}
induced by the adjunction $\id\to Rt_*t^{*}$, the fact that
$Rt_*=Rt_!$ and because $\pi\circ t=\id$. Together with the splitting 
principle from Corollary \ref{cor:splitting_sheaf} and the projection to the $k$-th component one gets an evaluation map 
\begin{equation}\label{eq:evaluation}
H^{2d-1}(S,R\pi_!Rj_{D*}j_D^{*}\Log{k}_\Qp(d))\xrightarrow{\varrho_t\circ t^{*}} H^{2d-1}(S,\prod_{i= 0}^{k}\Sym^{i}\sH_\Qp) \xrightarrow{\pr_k} H^{2d-1}(S,\Sym^{k}\sH_\Qp) .
\end{equation}
\begin{definition}\label{def:Qp-Eisenstein-class} Let $\alpha\in \Qp[D]$. 
The image of $_\alpha\cpol_\Qp$ under the evaluation map \eqref{eq:evaluation}
\[
_\alpha\Eis^{k}_{\Qp}(t)\in H^{2d-1}(S,\Sym^{k}\sH_\Qp) 
\]
is called the \emph{$k$-th \'etale $\Qp$-Eisenstein class for $G$}.
\end{definition}
\begin{remark}
The normalization in \cite[Definition 12.4.6]{Kings-Soule} is different. There we had an additional factor of $-N^{k-1}$ in front
of $_\alpha\Eis^{k}_{\Qp}(t)$. This has the advantage to make the residues of $_\alpha\Eis^{k}_{\Qp}(t)$ at the cusps integral, but is very unnatural from the point of view of the polylogarithm.
\end{remark}
Recall from \cite[Theorem 5.2.1]{HuKiPol} that the polylogarithm ${_\alpha}\cpol^{k}_\Qp$ is  
motivic, i.e., there exists a  class in motivic cohomology
\[
_\alpha\cpol_{\mot}^{k}\in H_\mot^{2d-1}(S,R\pi_!Rj_{D*}j_D^{*}\Log{k}_\mot(d)),
\]
the \emph{motivic polylogarithm},
which maps to $_\alpha\cpol^{k}{_\Qp}$ under the \'etale regulator  
\[
r_\et:H_\mot^{2d-1}(S,R\pi_!Rj_{D*}j_D^{*}\Log{k}_\mot(d))\to
H^{2d-1}(S,R\pi_!Rj_{D*}j_D^{*}\Log{k}_\Qp(d)).
\]
With the motivic  analogue of the evaluation map \eqref{eq:evaluation}
one can define exactly in the same way as in the \'etale case motivic 
Eisenstein classes for $\alpha\in \QQ[D]^{0}$
\begin{equation}
_\alpha\Eis_\mot^{k}(t)\in H^{2d-1}_\mot(S,\Sym^{k}\sH_\QQ). 
\end{equation}
The next proposition is obvious from the fact that the evaluation
map is compatible with the \'etale regulator.
\begin{proposition}
For $\alpha\in \QQ[D]^{0}$ the image of the motivic Eisenstein class
$_\alpha\Eis_\mot^{k}(t)$ under the \'etale  regulator
\[
r_\et:H^{2d-1}_\mot(S,\Sym^{k}\sH_\QQ)\to H^{2d-1}(S,\Sym^{k}\sH_\Qp)
\]
is the \'etale $\Qp$-Eisenstein class $_\alpha\Eis^{k}_{\Qp}(t)$.
\end{proposition}
\section{Sheaves of Iwasawa algebras}
\subsection{Iwasawa algebras}
Let $X=\prolim_r X_r$ be a profinite space with transition maps $\lambda_r:X_{r+1}\to X_r$
and 
\[
\Lambda_r[X_r]:=\Map(X_r,\ZZ/p^{r}\ZZ)
\]
the $\ZZ/p^{r}\ZZ$-module of maps from $X_r$ to $\ZZ/p^{r}\ZZ$.
For each $x_r$ we write $\delta_{x_r}\in \Lambda_r[X_r]$ for the map which is $1$ at $x_r$ and $0$ else. It is convenient to interpret $\Lambda_r[X_r]$ as the space
of $\ZZ/p^{r}\ZZ$-valued measures on $X_r$ and $\delta_{x_r}$ 
as the delta measure at $x_r$. Then the push-forward along  $\lambda_r:X_{r+1}\to X_r$ composed with reduction modulo $p^{r}$ induces $\Zp$-module maps
\begin{equation}\label{eq:Lambda-transition} 
\lambda_{r*}:\Lambda_{r+1}[X_{r+1}]\to \Lambda_{r}[X_{r}]
\end{equation}
which are characterized by $\lambda_{r*}(\delta_{x_{r+1}})=\delta_{\lambda_r(x_r)}$.
\begin{definition}
The \emph{module of $\Zp$-valued measures on $X$} is the inverse limit 
\[
\Lambda(X):=\prolim_r \Lambda_r[X_r]
\]
of $\Lambda_r[X_r]$ with respect to the transition maps from \eqref{eq:Lambda-transition}.
\end{definition}
Let $x=(x_r)_{r\ge 0}\in X$.  We define $\delta_x:=(\delta_{x_r})_{r\ge 0}\in \Lambda(X)$, which provides a map
\[
\delta:X\to \Lambda(X).
\]
For each continuous map $\varphi:X\to Y$ of profinite spaces
we get a homomorphism
\begin{equation}
\varphi_*:\Lambda(X)\to \Lambda(Y)
\end{equation}
"push-forward of measures" with the property $\varphi_*(\delta_x)=\delta_{\varphi(x)}$.
Obviously, one has 
$\Lambda_r[X_r\times Y_r]\isom\Lambda_r[X_r]\otimes\Lambda_r[Y_r]$ so that 
\[
\Lambda(X\times Y)\isom \Lambda(X)\widehat{\otimes}\Lambda(Y):=
\prolim_r \Lambda_r[X_r]\otimes\Lambda_r[Y_r].
\]
In particular, if $X=G=\prolim_r G_r$ is a profinite group, the
group structure $\mu:G\times G\to G$ induces a $\Zp$-algebra structure on $\Lambda(G)$, which coincides with the $\Zp$-algebra structure induced by the inverse limit of 
group algebras $\prolim_r \Lambda_r[G_r]$.
\begin{definition}
If $G=\prolim_r G_r$ is a profinite group, we call 
\[
\Lambda(G):=\prolim_r \Lambda_r[G_r]
\]
the \emph{Iwasawa algebra of $G$}. 
\end{definition}
More generally, if $G$ acts continuously on the profinite space $X$, one gets a map
\[
\Lambda(G)\widehat{\otimes}\Lambda(X)\to \Lambda(X)
\]
which makes $\Lambda(X)$ a $\Lambda(G)$-module.  If $X$ is a principal homogeneous space under $G$, then $\Lambda(X)$ is a free $\Lambda(G)$-module of rank $1$.

\subsection{Properties of the Iwasawa algebra}
In this section we assume that $H$ is a finitely generated free $\Zp$-module.  We let 
\[
H_r:=H\otimes_\Zp \Zp/p^{r}\Zp
\]
so that $H=\prolim_r H_r$ with the natural transition maps 
$H_{r+1}\to H_r$. 

In the case $H=\Zp$, the so called \emph{Amice transform} of a measure
$\mu\in \Lambda(\Zp)$
\[
\cA_\mu(T):=\sum_{n=0}^{\infty}T^{n}\int_{\Zp}{x\choose n}\mu(x)
\]
induces a ring isomorphism $\cA:\Lambda(\Zp)\isom \Zp[[T]]$
(see \cite[Section 1.1.]{Colmez-p-adic-L}). A straightforward generalization shows that $\Lambda(H)$ is isomorphic to
a power series ring in $\rk H$ variables. On the other hand one has the so called \emph{Laplace transform} of $\mu$ (see loc. cit.)
\[
\cL_\mu(t):=\sum_{n=0}^{\infty}\frac{t^{n}}{n!}\int_{\Zp}x^{n}\mu(x).
\]
This map is called the \emph{moment map} in \cite{Katz-p-adic-interpolation} and we will follow his terminology. In the next section, we will explain this map from an abstract algebraic point of view. For this we interpret $ \frac{t^{n}}{n!}$
as $ t^{[n]} $ in the divided power algebra $\Gamma_\Zp(\Zp)$.
\subsection{The moment map} We return to the case of a free 
$\Zp$-module $H$ of finite rank.
\begin{proposition}
Let $H$ be a free $\Zp$-module of finite rank and $H_r:=H\otimes_\Zp\ZZ/p^{r}\ZZ$. Then 
\[
\widehat{\Gamma}_\Zp(H)\isom \prolim_r \widehat{\Gamma}_{\ZZ/p^{r}\ZZ}(H_r).
\]
\end{proposition}
\begin{proof}
As each $\Gamma_\Zp(H)/\Gamma^{+}(H)^{[k]}$ is a finitely generated free $\Zp$-module, this follows by the compatibility with base change
of $\Gamma_\Zp(H)$ and the fact that one can interchange the inverse limits.
\end{proof}
By the universal property of the finite group ring $\Lambda_r[H_r]$, 
the group homomorphism
\begin{align*}
H_r&\to \widehat{\Gamma}_{\ZZ/p^{r}\ZZ}(H_r)^{\times}\\
h_r&\mapsto\sum_{k\ge 0}h_r^{[k]}
\end{align*}
induces a homomorphism of $\ZZ/p^{r}\ZZ$-algebras 
\[
\mom_r:\Lambda_r[H_r]\to \widehat{\Gamma}_{\ZZ/p^{r}\ZZ}(H_r).
\]
\begin{corollary}
The maps $ \mom_r $ induce in the inverse limit
a $\Zp$-algebra homomorphism
\[
\mom:\Lambda(H)\to \widehat{\Gamma}_{\Zp}(H).
\]
which is functorial in $H$.
\end{corollary}
\begin{definition}\label{def:classical-mom}
We call $\mom:\Lambda(H)\to \widehat{\Gamma}_{\Zp}(H)$ the \emph{moment map} and the composition with the projection to
$\Gamma_k(H)$ 
\[
\mom^{k}:\Lambda(H)\to \Gamma_k(H)
\]
the \emph{$k$-th moment map}.
\end{definition}
\subsection{Sheafification of the Iwasawa algebras}\label{sec:iwasawa}
Let $X$ be a separated noetherian scheme of finite type as in section \ref{section-Zp-sheaves} and $\sX:=(p_r:\sX_r\to X)_r$ be an inverse system of quasi-finite \'etale schemes over $X$ with \'etale transition maps $\lambda_r:\sX_r\to \sX_{r-1}$. We often write 
\begin{equation}
\Lambda_r:=\ZZ/p^{r}\ZZ.
\end{equation}
The adjunction $\lambda_{r!}\lambda_r^{!}\to \id$ defines a homomorphism
\[
p_{r+1!}\Lambda_{r+1}=p_{r!}\lambda_{r!}\lambda_r^{!}\Lambda_{r+1}\to 
p_{r!}\Lambda_{r+1},
\]
because $\lambda_r$ is \'etale. If one composes this with reduction modulo $p^{r}$, one gets a trace map
\begin{equation}\label{eq:trace-map}
\Tr_{r+1}:p_{r+1,!}\Lambda_{r+1}\to p_{r,!}\Lambda_{r}.
\end{equation}
\begin{definition}
We define an \'etale sheaf on $X$ by
\[
\Lambda_r[\sX_r]:=p_{r!}\Lambda_r.
\]
With the trace maps 
$\Tr_{r+1}:\Lambda_{r+1}[\sX_{r+1}]\to \Lambda_{r}[\sX_{r}]$ as transition morphisms we define the pro-sheaf
\[
\Lambda(\sX):=(\Lambda_r[\sX_r])_{r\ge 0}.
\]
\end{definition}
This definition is functorial in $\sX$. If $(\varphi_r)_r:(\sX_r)_r\to (\sY_r)_r$ is a morphism of inverse system of quasi-finite \'etale schemes over $X$, then the adjunction $\varphi_{r!}\varphi_r^{!}\to \id$ defines a morphism
\[
\varphi_{r!}:\Lambda_r[\sX_r]\to \Lambda_r[\sY_r]
\]
compatible with the transition maps, and hence 
a morphism of pro-sheaves
\[
\Lambda(\sX)\to \Lambda(\sY).
\]
Moreover, the formation of $\Lambda(\sX)$ is compatible with base change: if $\sX_{r,T}:=\sX_r\times_ST$ for an $S$-scheme
$f:T\to S$, then by proper base change one has
\[
f^{*}\Lambda_r[\sX_r]\isom \Lambda[\sX_{r,T}].
\]
By the K\"unneth formula, one has 
\[
\Lambda_r[\sX_r\times_X\sY_r]\isom \Lambda_r[\sX_r]\otimes \Lambda_r[\sY_r]
\] 
and hence $\Lambda(\sX\times_X\sY)\isom \Lambda(\sX)\widehat{\otimes}\Lambda(\sY)$ by taking the inverse limit.
In particular, in the case where $\sX=\sG$ is an inverse system of 
quasi-finite \'etale group schemes $\sG_r$, the group structure $\mu_r:\sG_r\times_X\sG_r\to \sG_r$ induces a ring structure 
\[
\Lambda(\sG)\widehat{\otimes}\Lambda(\sG)\to \Lambda(\sG)
\] 
on $\Lambda(\sG)$. Similarly, if 
\[
\sG\times_X\sX\to \sX
\]
is a group action of inverse systems, i.e., a compatible family of actions $\sG_r\times_X\sX_r\to \sX_r$, then $\Lambda(\sX)$ becomes a $\Lambda(\sG)$-module.

The next lemma shows that the above construction indeed sheafifies the Iwasawa algebras considered before. 

\begin{lemma}
Let $\overline{x}\in X$ be a geometric point and write 
$p_{r,\overline{x}}:\sX_{r,\overline{x}}\to \overline{x}$ for the base change of $\sX_r$ to 
$\overline{x}$ considered as a finite set. Then 
\[
\Lambda_r[\sX_r]_{\overline{x}}\isom\Lambda_r[X_r].
\]
\end{lemma}
\begin{proof}
This follows directly from the base change property of $\Lambda_r[\sX_r]$ and the fact that  $p_{r,\overline{x},!}\Lambda_r \isom\Lambda_r[X_r]$ over an algebraically closed field.
\end{proof}

We return to our basic set up, where $\pi:G\to S$ is a separated smooth
commutative group scheme with connected fibres. Recall from
\ref{lemma:description-of-H} that $\sH_r$ is the sheaf associated to $G[p^{r}]$, which is quasi-finite and \'etale over $S$. 
\begin{definition}
Define the \emph{sheaf of Iwasawa algebras} $\Lambda(\sH)$ on $S$ to be the pro-sheaf
\[
\Lambda(\sH):=(\Lambda_r[\sH_r])_{r\ge 0}.
\]
\end{definition}
\subsection{Sheafification of the moment map} We keep the notation of the previous section. In particular, we consider the \'etale sheaf $\sH_r$ and the sheaf $\Lambda_r[\sH_r]$.

Over $G[p^{r}]$ the sheaf $[p^{r}]^{*}\sH_r$ has the tautological section $\tau_r\in \Gamma(G[p^{r}],[p^{r}]^{*}\sH_r)$ corresponding to the
identity map $G[p^{r}]\to \sH_r$. This gives rise to the section
\begin{equation}\label{eq:tau}
\tau_r^{[k]}\in \Gamma(G[p^{r}],[p^{r}]^{*}\Gamma_k(\sH_r))
\end{equation}
of the $k$-th divided power of $\sH_r$.
Using the chain of isomorphisms (note that $[p^{r}]^{*}=[p^{r}]^{!}$ as $[p^{r}]$ is \'etale)
\[
\Gamma(G[p^{r}],[p^{r}]^{*}\Gamma_k(\sH_r))\isom
\Hom_{G[p^{r}]}(\ZZ/p^{r}\ZZ,[p^{r}]^{*}\Gamma_k(\sH_r) )\isom
\Hom_S([p^{r}]_!\ZZ/p^{r}\ZZ,\Gamma_k(\sH_r)),
\]
the section $\tau_r^{[k]}$ gives rise to a morphism of sheaves
\begin{equation}\label{eq:mom-k-r}
\mom^{k}_r:\Lambda_r[\sH_r]\to \Gamma_k(\sH_r).
\end{equation}
\begin{lemma} There is a commutative diagram 
\[
\begin{CD}
\Lambda_r[\sH_r]@>\mom_r^{k}>>\Gamma_k(\sH_r)\\
@V\Tr_rVV@VVV\\
\Lambda_{r-1}[\sH_{r-1}]@>\mom_{r-1}^{k}>>\Gamma_k(\sH_{r-1})
\end{CD}
\]
where the right vertical map is given by the reduction map
\[
\Gamma_k(\sH_r)\to \Gamma_k(\sH_r)\otimes_{\ZZ/p^{r}\ZZ}\ZZ/p^{r-1}\ZZ\isom
\Gamma_k(\sH_{r-1}).
\]
\end{lemma}
\begin{proof} Denote by $\lambda_r:\sH_r\to \sH_{r-1}$ the  transition map. Reduction modulo $p^{r-1}$ gives a commutative diagram
\[
\begin{CD}
[p^{r}]_!\ZZ/p^{r}\ZZ@>\mom^{k}_r >> \Gamma_k(\sH_r)\\
@VVV@VVV\\
[p^{r}]_!\lambda_r^{*}\ZZ/p^{r-1}\ZZ@>\mom^{k}_r\otimes \ZZ/p^{r-1}\ZZ>> \Gamma_k(\sH_{r-1}).
\end{CD}
\]
As the image of the tautological class $\tau_r^{[k]}\in \Gamma(G[p^{r}],[p^{r}]^{*}\Gamma_k(\sH_r))$ under the reduction map gives the the pull-back of the tautological class 
\begin{align*}
\lambda_r^{*}\tau^{[k]}_{r-1}\in \Gamma(G[p^{r}],[p^{r}]^{*}\Gamma_k(\sH_{r-1}))&\isom
\Hom_{G[p^{r}]}(\lambda_r^{*}\ZZ/p^{r-1}\ZZ,[p^{r}]^{*}\Gamma_k(\sH_{r-1}))\\
&\isom \Hom_{S}([p^{r}]_!\lambda_r^{*}\ZZ/p^{r-1}\ZZ,\Gamma_k(\sH_{r-1}))
\end{align*}
one concludes that $\mom^{k}_r\otimes \ZZ/p^{r-1}\ZZ$ coincides with the map given by $\lambda_r^{*}\tau^{[k]}_{r-1}$. This means that 
$\mom^{k}_r\otimes \ZZ/p^{r-1}\ZZ$ has to factor through $\Tr_r$, i.e., the diagram
\[
\xymatrix{
[p^{r-1}]_!\lambda_{r!}\lambda_r^{*}\ZZ/p^{r-1}\ZZ
\ar[rr]^{\mom^{k}_r\otimes \ZZ/p^{r-1}\ZZ}\ar[dr]_{\Tr_r}&&\Gamma_k(\sH_{r-1})\\
&[p^{r-1}]_!\ZZ/p^{r-1}\ZZ\ar[ur]_{\mom_{r-1}^{k}}
}
\]
commutes, which gives the desired result.
\end{proof}
With this result we can now define the moment map for the 
sheaf of Iwasawa algebras $\Lambda(\sH)$. 
\begin{definition}\label{def:mom}
We define the \emph{$k$-th moment map } to be the map of pro-sheaves
\[
\mom^{k}:\Lambda(\sH)\to \Gamma_k(\sH)
\]
defined by $(\mom_r^{k})_{r\ge 0}$ and 
\[
\mom:\Lambda(\sH)\to \widehat{\Gamma}_\Zp(\sH)
\]
by taking $\mom^{k}$ in the $k$-th component.
\end{definition}
\begin{remark}
In each stalk the the map $\mom^{k}$ coincides with the map $\mom^{k}$ defined in \ref{def:classical-mom} (see \cite[Lemma 12.2.14]{Kings-Soule}).
\end{remark}
\section{The  integral logarithm sheaf}
\subsection{Definition of the integral logarithm sheaf}
We now define a pro-sheaf $\cL$ on $G$ of modules over $\pi^{*}\Lambda(\sH)$, which will give a $\Zp$-structure of the
logarithm sheaf $\cLog_\Qp$. For this write $G_r:=G$ considered as 
a quasi-finite \'etale $G$-scheme via the $p^{r}$-multiplication
\begin{equation}
[p^{r}]:G_r=G\to G.
\end{equation}
Note that this is a $G[p^{r}]$-torsor
\[
0\to G[p^{r}]\to G_r\xrightarrow{[p^{r}]}G\to 0
\]
over $G$.
Let $\lambda_r:G_r\to G_{r-1}$ be the transition map, 
which is just the $[p]$-multiplication in this case. Then, as 
in \eqref{eq:trace-map}, we have trace maps 
\[
\Tr_r:\Lambda_r[G_r]\to \Lambda_{r-1}[G_{r-1}].
\]
We will also need the following variant. Let $\Lambda_s:=\ZZ/p^{s}\ZZ$ and write
\begin{equation}
\Lambda_s[G_r]:=[p^{r}]_!\Lambda_s.
\end{equation}
Then the adjunction $\lambda_{r!}\lambda_r^{!}\to \id$ defines transition morphisms
\begin{equation}
\lambda_{r!}:\Lambda_s[G_r]\to \Lambda_s[G_{r-1}].
\end{equation}
\begin{definition}
With the above transition maps we can define the pro-sheaves
\begin{align*}
\cL:=(\Lambda_r[G_r])_{r\ge 0}&&
\mbox{and}
&& \cL_{\Lambda_s}:=(\Lambda_s[G_r])_{r\ge 0}.
\end{align*}
We call $\cL$ the \emph{integral logarithm sheaf}.
\end{definition}
Note that the reduction modulo $p^{s-1}$ gives transition maps
$\cL_{\Lambda_s}\to \cL_{\Lambda_{s-1}}$ and that we have an isomorphism of pro-sheaves
\begin{equation}
\cL\isom (\cL_{\Lambda_s})_{s\ge 0}.
\end{equation}
By the general theory outlined above, $\cL$ is a module over
$\pi^{*}\Lambda(\sH)$ which is free of rank $1$. 

Let $t:S\to G$ be a section and denote by $G[p^{r}]\langle t\rangle$
the  $G[p^{r}]$-torsor defined by the cartesian diagram
\begin{equation}\label{eq:G-torsor-defn}
\begin{CD}
G[p^{r}]\langle t\rangle@>>> G_r\\
@VVV@VV[p^{r}] V\\
S@>t>> G.
\end{CD}
\end{equation}
We denote by $\sH_r\langle t\rangle$ the \'etale sheaf defined by 
$G[p^{r}]\langle t\rangle$ and by 
$\sH\langle t\rangle:=(\sH_r\langle t\rangle)$ the pro-system defined 
by the trace maps. 
We write 
\[
\Lambda(\sH\langle t\rangle):=(\Lambda_r[\sH_r\langle t\rangle])_{r\ge 0}
\]
for the sheaf of Iwasawa modules defined by $\sH\langle t\rangle$.
\begin{lemma}\label{lemma:base-change-for-cL} There is an canonical isomorphism
\[
t^{*}\cL\isom \Lambda(\sH\langle t\rangle).
\]
In particular, for the unit section $e:S\to G$ one has
\begin{align*}
e^{*}\cL\isom \Lambda(\sH)
\end{align*}
and hence a section $\bfone:\Zp\to e^{*}\cL$ given by mapping $1$ to $1$.
\end{lemma}
\begin{proof}
This follows directly from the fact that $\cL$ is compatible with base change and the definitions.
\end{proof}
\subsection{Basic properties of the integral logarithm sheaf}
The integral logarithm sheaf enjoys the same properties as its 
$\Qp$-counterpart, namely functoriality, vanishing of cohomology and 
a universal property for unipotent sheaves.

Let $\varphi:G_1\to G_2$
be a homomorphism of group schemes of relative 
dimension $d_1$ and $d_2$ over $S$. Denote by $\cL_1$ and $\cL_2$ the
integral logarithm sheaves on $G_1$ and $G_2$ respectively.
\begin{theorem}[Functoriality]
Let $c:=d_1-d_2$. Then there is a canonical map 
\[
\varphi_\#:\cL_1\to \varphi^{*}\cL_2\isom \varphi^{!}\cL_2(-c)[-2c].
\]
Moreover, if $\varphi$ is an isogeny of degree prime to $p$, then
$\varphi_\#:\cL_1\isom \varphi^{*}\cL_2$ is an isomorphism.
\end{theorem}
\begin{proof}
The homomorphism $\varphi$ induces a homomorphism of group schemes over $G_1$
\begin{equation}\label{eq:functoriality-varphi}
\varphi:G_{1,r}\to G_{2,r}\times_{G_2}G_1
\end{equation}
which induces by adjunction $\varphi_!\varphi^{!}\to \id$ and
the base change property of $\Lambda_{r}[G_{2,r}]$
a morphism of sheaves 
\[
\varphi_\#:\Lambda_r[G_{1,r}]\to \varphi^{*}\Lambda_{r}[G_{2,r}]=\varphi^{!} \Lambda_{r}[G_{2,r}]  (-c)[-2c].
\]
Passing to the limit gives the required map. If $\varphi$ is an isogeny of degree prime to $p$, then the map in \eqref{eq:functoriality-varphi} is an isomorphism. Hence this is also
true for $\varphi_\#$.
\end{proof}
\begin{corollary}[Splitting principle] Let $c$ be an integer prime to $p$ and
let $t:S\to G$ be a $c$-torsion section. Then 
there is an isomorphism
\[
[c]_\#:t^{*}\cL\isom \Lambda(\sH).
\]
More generally, if $D:=G[c]$ with $(c,p)=1$ then 
\[
\iota_D^{*}\cL\isom \pi_{D}^{*}\Lambda(\sH),
\]
where $\iota_D:D\to G$ and $\pi_D:D\to S$ is the structure map.
\end{corollary}
\begin{proof}
Apply $t^{*}$ respectively, $\iota_D^{*}$ to the isomorphism $[c]_\#:\cL\to [c]^{*}\cL$.
\end{proof}
\begin{theorem}[Vanishing of cohomology]\label{thm:coh-computation-of-cL}
Recall that $2d$ is the relative dimension of $\pi:G\to S$. Then the pro-sheaves 
\[
R^{i}\pi_!\cL\mbox{ for $i<2d$}
\]
are Mittag-Leffler zero (see \ref{section-Zp-sheaves}) and 
\[
R^{2d}\pi_!\cL(d)\isom \Zp.
\]
\end{theorem}
We start the proof of this theorem with a lemma:
\begin{lemma} \label{lemma:trivializing-coh}
The endomorphism $[p^{r}]_!:R^{i}\pi_!\ZZ/p^{s}\ZZ\to R^{i}\pi_!\ZZ/p^{s}\ZZ$ is given by multiplication with $p^{r(2d-i)}$.
\end{lemma}
\begin{proof}
By Lemma \ref{lemma:description-of-H} we see that $[p^{r}]_!$ is given by $p^{r}$-multiplication on $\sH_s$. The result follows from this and the
$\ZZ/p^{s}\ZZ$-version of the isomorphism \eqref{eq:Hopf-alg}
\end{proof}

\begin{proof}[Proof of Theorem \ref{thm:coh-computation-of-cL}]
Consider the transition map $\Lambda_s[G_{r+j}]\to \Lambda_s[G_r]$.
If we apply $R^{i}\pi_!$ we get the homomorphism
\[
[p^{j}]_!:R^{i}\pi_{r+j,!}\Lambda_s\to R^{i}\pi_{r,!}\Lambda_s,
\]
where $\pi_r=\pi:G_r\to S$ is the structure map of $G_r=G$. By Lemma
\ref{lemma:trivializing-coh}, the map $[p^{j}]_!$ acts by multiplication with $p^{j(2d-i)}$ on $R^{i}\pi_{r+j,!}\Lambda_s$. In particular, this is zero for $i\neq 2d$ and $j\ge s$ and the identity for $i=2d$. This proves the theorem, because $R^{2d}\pi_!\Lambda_s(d)\isom \Lambda_s$.
\end{proof}
The sheaf $\cL$ satisfies also a property analogous to Theorem 
\ref{thm:sheaf-univ-property}. To formulate this properly, we first need a
property of unipotent $\ZZ/p^{s}\ZZ$-sheaves. 
\begin{lemma}\label{lemma:unipotent-trivialization}
Let $\sF$ be a unipotent $\Lambda_s=\ZZ/p^{s}\ZZ$-sheaf of length $n$ on $G$. Then $[p^{ns}]^{*}\sF$ is trivial on $G_{ns}$ in the sense that there exists a $\Lambda_s$-sheaf $\sG$ on $S$ such that
\[
[p^{ns}]^{*}\sF\isom \pi_{ns}^{*}\sG,
\]
where $\pi_{ns}:G_{ns}\to S$ is the structure map.
\end{lemma}
\begin{proof}
We show this by induction. For $n=0$ there is nothing to show.
So let 
\[
0\to \sF'\to \sF\to \pi^{*}\sG''\to 0
\]
be an exact sequence with $\sF'$ unipotent of length $n-1$, so
that by induction hypotheses $[p^{(n-1)s}]^{*}\sF'\isom \pi^{*}\sG'$
on $G_{(n-1)s}$. Thus it suffices to show that for an extension
$\sF\in \Ext^{1}_{G}(\pi^{*}\sG'',\pi^{*}\sG')$, the sheaf 
$[p^{s}]^{*}\sF$ is trivial on $G_s$. One has 
\begin{align*}
\Ext^{1}_{G}(\pi^{*}\sG'',\pi^{*}\sG')\isom \Ext^{1}_{G}(\pi^{!}\sG'',\pi^{!}\sG')\isom \Ext^{1}_{S}(R\pi_!\pi^{!}\sG'',\sG')
\end{align*}
and the pull-back by $[p^{s}]^{*}$ on the first group is induced 
by the trace map $[p^{s}]_!:R\pi_![p^{s}]_![p^{s}]^{!}\pi^{!}\sG''\to R\pi_!\pi^{!}\sG''$ on the last group.
By the projection formula we have $R\pi_!\pi^{!}\sG''\isom R\pi_!\Lambda_s(d)[2d]\otimes \sG''$ and the triangle
\[
\tau_{<2d}R\pi_!\Lambda_s(d)[2d]\to R\pi_!\Lambda_s(d)[2d]\to R^{2d}\pi_!\Lambda_s(d)\isom \Lambda_s
\]
gives rise to  a long exact sequence of $\Ext$-groups
\[
\ldots\to\Ext^{1}_{S}(\sG'',\sG')\to
\Ext^{1}_{S}(R\pi_!\Lambda_s(d)[2d]\otimes\sG'',\sG')\to 
\Ext^{1}_{S}(\tau_{<2d}R\pi_!\Lambda_s(d)[2d]\otimes\sG'',\sG')\to \ldots
\]
If we pull-back by $[p^{s}]^{*}$ and use Lemma \ref{lemma:trivializing-coh} the resulting map on 
$\Ext^{1}_{S}(\tau_{<2d}R\pi_!\Lambda_s(d)[2d]\otimes\sG'',\sG')$ is 
zero, which shows that $[p^{s}]^{*}\sF$ is in the image of 
\[
\Ext^{1}_{S}(\sG'',\sG')\xrightarrow{[p^{s}]^{*}\pi^{*}}\Ext^{1}_{G}([p^{s}]^{*}\pi^{*}\sG'',[p^{s}]^{*}\pi^{*}\sG').
\]
This is the desired result.
\end{proof}
Exactly as in \eqref{eq:universal-property-map} one can define for each $\Lambda_s$-sheaf $\sF$ and 
each $r$ a homomorphism
\begin{equation}\label{eq:evaluation-map}
\pi_*\underline\Hom_G(\Lambda_s[G_r],\sF)\to e^*\sF 
\end{equation}
as the composition  
\[
\pi_*\underline{\Hom}_G(\cL_{\Lambda_s,r},\sF)\to \pi_*e_*e^{*}\underline{\Hom}_G(\cL_{\Lambda_s,r},\sF)
\to \underline{\Hom}_S(e^*\cL_{\Lambda_s,r},e^*\sF)\xrightarrow{{\bfone}^{*}}
\underline{\Hom}_S(\Lambda_s,e^*\sF)
\]
The next theorem corrects and generalizes 
\cite[Proposition 4.5.3]{Kings-Soule},
which was erroneously stated for all $\ZZ/p^{s}\ZZ$-sheaves and not just for unipotent ones.
\begin{theorem}[Universal property]\label{thm:integral-universal-property}
Let $\sF$ be a unipotent $\Lambda_s$-sheaf of length $n$. Then the homomorphism \eqref{eq:evaluation-map}
\[
\pi_*\underline\Hom_G(\Lambda_s[G_{ns}],\sF)\isom e^*\sF 
\]
is an isomorphism.
\end{theorem}
\begin{proof}
Let $\sF$ be unipotent of length $n$. Then we know from 
Lemma \ref{lemma:unipotent-trivialization} that there is 
a $\Lambda_s$-sheaf $\sG$ on $S$ such that $[p^{ns}]^{*}\sF\isom \pi^{*}_{ns}\sG$, where $\pi_{ns}:G_{ns}\to S$ is the structure map. Similarly, we write $e_{ns}$ for the unit section of $G_{ns}$. Then one has
\[
e^{*}\sF\isom e^{*}_{ns}[p^{ns}]^{*}\sF\isom e^{*}_{ns}\pi_{ns}^{*}\sG\isom \sG.
\]
Further, one has the following chain of isomorphisms
\begin{align*}
\pi_*\underline{\Hom}_{G}(\Lambda_s[G_{ns}],\sF)=
\pi_*\underline{\Hom}_{G}([p^{ns}]_!{\Lambda_s},\sF)&\isom
\pi_{ns*}\underline{\Hom}_{G_{ns}}(\Lambda_s,[p^{ns}]^{*}\sF)\\
&\isom \pi_{ns*}\underline{\Hom}_{G_{ns}}(\Lambda_s,\pi_{ns}^{*}\sG)\\
&\isom\underline{\Hom}_{S}(R\pi_{ns!}\Lambda_s(d)[2d],\sG)\\
&\isom\underline{\Hom}_{S}(R^{2d}\pi_{ns!}\Lambda_s(d),\sG)\\
&\isom \sG\isom e^{*}\sF,
\end{align*}
which prove the theorem.
\end{proof}
\subsection{The integral \'etale poylogarithm}
In this section we define in complete analogy with the $\Qp$-case the
integral \'etale polylogarithm. 

We recall the set-up from section
\ref{sec:Qp-polylog}. Denote by $c>1$ an integer invertible on $S$ and
prime to $p$ and 
let $D:=G[c]$ be the $c$-torsion subgroup. Then the  localization triangle for $j_D:U_D\subset G$ and $\iota_D:D\to G$ reads
\[
R\pi_!\cL(d)[2d-1]\to R\pi_!Rj_{D*}j_D^{*}\cL(d)[2d-1]\to \pi_{D!}\iota_D^{!}\cL(d).
\]
By relative purity and the splitting principle $\iota_D^{!}\cL(d)[2d]\isom \iota_D^{*}\cL\isom \pi_D^{*}\Lambda(\sH)$. 
We apply the functor $H^{j}(S,-)$ to this triangle. As the
$R^{i}\pi_!\cL$ are Mittag-Leffler zero for $i\neq 2d$
by Theorem \ref{thm:coh-computation-of-cL} one gets with \eqref{eq:vanishing-MLzero}:
\begin{proposition}
In the above situation there is a short exact sequence
\[
0\to H^{2d-1}(S,R\pi_!Rj_{D*}j_D^{*}\cL(d))\xrightarrow{\res} H^{0}(S,\pi_{D!}\pi_D^{*}\Lambda(\sH))\to H^{0}(S,\Zp)\to 0.
\]
\end{proposition}
As in the $\Qp$-case we define 
\[
\Zp[D]^{0}:=\ker\left( H^{0}(S,\pi_{D!}\pi_D^{*}\Zp)\to H^{0}(S,\Zp)\right)
\]
so that one has 
\[
\Zp[D]^{0}\subset \ker\left(H^{0}(S,\pi_{D!}\pi_D^{*}\Lambda(\sH))\to H^{0}(S,\Zp)\right).
\]
With these preliminaries we can define the integral polylogarithm.
\begin{definition}
The \emph{integral \'etale polylogarithm}  associated to $\alpha\in \Zp[D]^{0}$ is the unique class 
\[
_\alpha\cpol\in H^{2d-1}(S,R\pi_!Rj_{D*}j_D^{*}\cL(d))
\]
such that $\res(_\alpha\cpol)=\alpha$.
\end{definition}

\subsection{The Eisenstein-Iwasawa class} Recall that $D=G[c]$ and
let $t:S\to U_D=G\setminus D$ be an $N$-torsion section with $(N,c)=1$
but $N$ not necessarily prime to $p$. The same chain of maps as
in $\eqref{eq:t-evaluation}$ gives a map
\begin{equation}\label{eq:integral-t-evaluation}
H^{2d-1}(S,R\pi_!Rj_{D*}j_D^{*}\cL(d))\to H^{2d-1}(S,t^{*}\cL(d))\isom H^{2d-1}(S,\Lambda(\sH\langle t\rangle)(d)).
\end{equation}
By functoriality the $N$-multiplication induces a homomorphism
\[
[N]_\#:\Lambda(\sH\langle t\rangle)\to\Lambda(\sH).
\]
\begin{definition}\label{def:Eisenstein-Iwasawa}Let $\alpha\in \Zp[D]^{0}$ and $t:S\to U_D$ be an $N$-torsion section. Then the
image
\[
_\alpha\EI(t)\in H^{2d-1}(S, \Lambda(\sH\langle t\rangle)(d))
\]
of $_\alpha\cpol$ under the map  \eqref{eq:integral-t-evaluation} is
called the \emph{Eisenstein-Iwasawa class}. We write 
\[
_\alpha\EI(t)_N:=[N]_\#(_\alpha\EI(t))\in H^{2d-1}(S, \Lambda(\sH)(d)).
\]
\end{definition}
\begin{remark}
Note that $_\alpha\EI(t)_N$ depends on $N$ and not on $t$ alone. The
class $_\alpha\EI(t)_{NM}$ differs from $_\alpha\EI(t)_N$.
\end{remark}
The $k$-th moment map induces a homomorphism of cohomology groups
\begin{equation}
\mom^{k}:H^{2d-1}(S, \Lambda(\sH)(d))\to H^{2d-1}(S,\Gamma_k(\sH)(d)).
\end{equation}
\begin{definition}
The class
\[
_\alpha\Eis^{k}_N(t):=\mom^{k}({_\alpha}\EI_N)\in H^{2d-1}(S,\Gamma_{k}(\sH)(d))
\]
is called the \emph{integral \'etale Eisenstein class}.
\end{definition}
These Eisenstein classes are interpolated by the Eisenstein-Iwasawa class by definition. We will see later how they are related to the $\Qp$-Eisenstein class, which are motivic, i.e., in the image of the \'etale regulator from motivic cohomology.
\subsection{The Eisenstein-Iwasawa class for abelian schemes}
\label{sec:abelian-schemes}
It is worthwhile to consider the case of abelian schemes in more detail. In this section we let $G=A$ be an abelian scheme over $S$, so that in particular $\pi:A\to S$ is proper and we can write $R\pi_*$ instead of $R\pi_!$. 

The first thing to  observe is the isomorphism 
\[
H^{2d-1}(S,R\pi_!Rj_{D*}j_D^{*}\cLog(d))\isom H^{2d-1}(U_D,\cLog(d)),
\]
so that the $\Qp$-polylogarithm is a class 
\[
_\alpha\cpol_\Qp\in H^{2d-1}(U_D,\cLog(d)).
\]
Evaluation at the $N$-torsion section $t:S\to U_D$ is just the pull-back with $t^{*}$
\[
t^{*}_\alpha\cpol_\Qp\in H^{2d-1}(S,t^{*}\cLog(d))\isom H^{2d-1}(S,\prod_{k\ge 0}\Sym^{k}\sH_\Qp(d))
\]
and the $k$-th component of $t^{*}_\alpha\cpol_\Qp$ is $_\alpha\Eis_\Qp^{k}(t)$. 

There is one specific choice of $\alpha$ which is particularly important, which we define next. Consider the finite \'etale morphism $\pi_D:G[c]\to S$  and the unit section $e:S\to G[c]$. These induce
\[
e_*:H^{0}(S,\Qp)\to H^{0}(S,\pi_{D*}\Qp)
\]
(coming from $\pi_{D*}e_!e^{!}\Qp\to\pi_{D*}\Qp$) and 
\[
\pi_D^{*}:H^{0}(S,\Qp)\to H^{0}(S,\pi_{D*}\Qp).
\]
One checks easily that $e_*(1)-\pi^{*}_D(1)$ is in the kernel of 
$H^{0}(S,\pi_{D*}\Qp)\to H^{0}(S,\Qp)$.
\begin{definition}
Let $\alpha_c\in \Qp[D]^{0}$ be the class
\[
\alpha_c:=e_*(1)-\pi^{*}_D(1).
\]
We write $_c\cpol_\Qp$ and $_c\Eis^{k}_\Qp(t)$ for the polylogarithm
and the Eisenstein class defined with $\alpha_c$.
\end{definition} 
We now assume that $S$ is of finite type over $\Spec\ZZ$. Then 
$H^{2d-2}(A_r\setminus A_r[cp^{r}],\ZZ/p^{r}\ZZ(d))$ is finite, so that one has by \eqref{eq:cont-coh}
\[
H^{2d-1}(S,R\pi_!Rj_{D*}j_D^{*}\cL(d))\isom H^{2d-1}(A\setminus A[c],\cL(d))\isom
\prolim_r H^{2d-1}(A_r\setminus A_r[cp^{r}],\ZZ/p^{r}\ZZ(d))
\]
where, as before, $[p^{r}]:A_r=A\to A$ is the $p^{r}$-multiplication and  the transition maps are given by the trace maps. The integral \'etale polylogarithm is then a class
\[
_\alpha\cpol\in 
\prolim_r H^{2d-1}(A_r\setminus A_r[cp^{r}],\ZZ/p^{r}\ZZ(d)).
\]
In the special case where $A=E$ is an elliptic curve over $S$ it is
shown in \cite[Theorem 12.4.21]{Kings-Soule} that
\[
_c\cpol\in 
\prolim_r H^{1}(E_r\setminus E_r[cp^{r}],\ZZ/p^{r}\ZZ(d))
\]
is given by the inverse limit of  Kato's norm compatible elliptic units $_c\vartheta_E$. Unfortunately, we do not have such a description even in the case of abelian varieties of dimension $\ge 2$.
If we write $A[p^{r}]\langle t\rangle$ for  
the $A[p^{r}]$-torsor defined by diagram \eqref{eq:G-torsor-defn}, then
\[
_\alpha\EI(t)\in H^{2d-1}(S,t^{*}\cL(d))=\prolim_r H^{2d-1}(A[p^{r}]\langle t\rangle,\ZZ/p^{r}\ZZ(d))  
\]
where the inverse limit is again over the trace maps. 
\section{Interpolation of the $\Qp$-Eisenstein classes}
\subsection{An integral structure on $\Log{k}_\Qp$}
For the comparison between the integral $\cL$ and the $\Qp$-polylogarithm $\cLog_{\Qp}$ we need an intermediate object, which we define in this section. This is purely technical. The reason for 
this is as follows:
In general a unipotent $\Qp$-sheaf does not necessarily have a $\Zp$-lattice which is again a unipotent sheaf. In the case of 
$\Log{k}_\Qp$ however, it is even possible to construct a $\Zp$-structure
$\Log{k}$ such that 
\[
\Log{k}_{\Lambda_r}:=\Log{k}\otimes_\Zp \Lambda_r
\]
is a unipotent $\Lambda_r=\ZZ/p^{r}\ZZ$-sheaf.

Let $\Log{1}$ be the $\Zp$-sheaf defined in \ref{defn:log}
\begin{equation}
0\to\sH\to \Log{1}\to \Zp\to 0
\end{equation}
and denote by $\one{1}:\Zp\to e^{*}\Log{1}$ a fixed splitting.

\begin{definition}\label{def:integral-log}
We define
\[
\Log{k}:=\Gamma_k(\Log{1})
\]
as the $k$-th graded piece of 
the divided power algebra $\Gamma_\Zp(\Log{1})$. We 
further denote by 
\[
\one{k}:=\Gamma_k(\one{1}):\Zp\to \Log{k}
\]
the splitting induced by $\one{1}$.
\end{definition}

As $\Zp$ and $\sH$ are flat $\Zp$-sheaves (all stalks are $\Zp$-free), the $k$-th graded piece of 
the divided power algebra $\Gamma_k(\Log{1})$ has a filtration 
with graded pieces $\pi^{*}\Gamma_i(\sH)\otimes\Gamma_{k-i}(\Zp)$ (see \cite[V 4.1.7]{Illusie}). In particular, the $\Gamma_k(\Log{1})$ are unipotent
$\Zp$-sheaves of length $k$. By base change the same is true for the 
$\Lambda_r$-sheaf
\begin{equation}\label{eq:Lambda-r-Log}
\Log{k}_{\Lambda_r}:=\Log{k}\otimes_\Zp\Lambda_r.
\end{equation}
To define  transition maps 
\begin{equation}
\Log{k}\to \Log{k-1}
\end{equation}
we proceed as in Section \ref{sec:Qp-logarithm}. Consider $\Log{1}\to \Zp\oplus\Log{1}$ given by the canonical projection and the identity. Then we define
\begin{multline*}
\Log{k}=\Gamma_k(\Log{1})\to \Gamma_k(\Zp\oplus\Log{1})\isom
\bigoplus_{i+j=k}\Gamma_i(\Zp)\otimes\Gamma_j(\Log{1})\to\\
\to
\Gamma_1(\Zp)\otimes\Gamma_{k-1}(\Log{1})\isom \Log{k-1}
\end{multline*}
where we identify $\Gamma_1(\Zp)\isom\Zp$. A straightforward computation shows that $\one{k}\mapsto \one{k-1}$ under the transition map. 
\begin{definition}
We denote by $\cLog$ the pro-sheaf $(\Log{k})_{k\ge 0}$ with the above 
transition maps and let $\bfone:\Zp\to e^{*}\cLog$ be the splitting
defined by $(\one{k})_{k\ge 0}$.
\end{definition}
\begin{remark}
We would like to point out that, contrary to the $\Qp$-situation,
the pro-sheaf $(\Log{k})_{k\ge 0}$ is \emph{not} the correct
definition of the $\Zp$-logarithm sheaf. In fact, the correct integral logarithm sheaf is $\cL$. 
\end{remark}
\begin{proposition}\label{prop:integral-structure-of-log}
Denote by $\Log{k}\otimes \Qp$ the $\Qp$-sheaf associated to $\Log{k}$.
Then there is a canonical isomorphism
\[
\Log{k}_\Qp\isom \Log{k}\otimes \Qp
\]
which maps $\one{k}_\Qp$ to $\one{k}$.
\end{proposition}
\begin{proof} First note that the canonical map $\Sym^{k}\Log{1}_\Qp\to \Gamma_k(\Log{1}_\Qp)$ is an isomorphism.
This can be checked at stalks, where it follows from \eqref{eq:Sym-und-Gamma} as $\Log{1}_\Qp$ is a sheaf of $\Qp$-modules. The claim in the proposition then follows from the isomorphisms
\[
\Log{k}_\Qp=\Sym^{k}\Log{1}_\Qp\isom \Gamma_k(\Log{1}_\Qp)\isom 
\Gamma_k(\Log{1})\otimes\Qp=\Log{k}\otimes\Qp
\]
and the claim about the splitting follows from the explicit formula for the map $\Sym^{k}\Log{1}_\Qp\to \Gamma_k(\Log{1}_\Qp)$ given after \eqref{eq:Sym-und-Gamma}.
\end{proof}
\begin{corollary}\label{cor:comp-isom}
For all $i$ there are isomorphisms
\begin{align*}
H^{i}(S,R\pi_!Rj_{D*}j_D^{*}\Log{k}(d))\otimes_\Zp\Qp &\isom 
H^{i}(S,R\pi_!Rj_{D*}j_D^{*}\Log{k}_\Qp(d))\\
H^{i}(S,\pi_{D!}\pi_D^{*}\prod_{i=0}^{k}\Gamma_i(\sH))\otimes_\Zp\Qp &\isom 
H^{i}(S,\pi_{D!}\pi_D^{*}\prod_{i=0}^{k}\Sym^{i}\sH_\Qp)\\
H^{i}(S,R\pi_!\Log{k}(d)[2d])\otimes_\Zp\Qp &\isom H^{i}(S,R\pi_!\Log{k}_\Qp(d)[2d])
\end{align*}
\end{corollary}
\begin{proof}
The first and the third follow directly from the proposition and the definition of the cohomology of a $\Qp$-sheaf. For the second one observes that the canonical map 
\[
\Sym^{k}\sH_{\Qp}\isom \Sym^{k}\sH\otimes\Qp\to\Gamma_k(\sH)\otimes\Qp\isom \Gamma_k(\sH_\Qp)
\] 
is an isomorphism. This can be checked on stalks, where it follows again from \eqref{eq:Sym-und-Gamma}.
\end{proof}
\subsection{Comparison of integral and $\Qp$-polylogarithm}
In this section we want to compare $\cL$ and $\cLog_{\Qp}$. We first
compare $\cL$ with the  sheaves $\Log{k}$ defined in \ref{def:integral-log}.  

Define a comparison map
\begin{equation*}
\comp^{k}:\cL\to \Log{k}
\end{equation*}
as follows. By Theorem 
\ref{thm:integral-universal-property} one has for the sheaves
$\Log{k}_{\Lambda_r}$ from \eqref{eq:Lambda-r-Log} the isomorphism
\[
\Hom_G(\Lambda_r[G_{rk}],\Log{k}_{\Lambda_r})\isom H^{0}(S,e^{*}\Log{k}_{\Lambda_r}),
\]
so that the splitting $\one{k}\otimes \Lambda_r:\Lambda_r\to e^{*}\Log{k}_{\Lambda_r}$ defines a morphism of sheaves on $G$
\begin{equation}\label{eq:comp-map}
\comp_r^{k}:\Lambda_r[G_{rk}]\to \Log{k}_{\Lambda_r},
\end{equation}
which is obviously compatible with the transition maps and functorial in $G$. Passing to the 
pro-systems over $r\ge 0$, this defines a homomorphism
\begin{equation}
\comp^{k}: \cL\to \Log{k}.
\end{equation}
Taking also the pro-system in the $k$-direction leads to a comparison map
\begin{equation}
\comp: \cL\to \cLog.
\end{equation}
For each $k$ applying $\comp^{k}$ to the localization triangle for $D\hookrightarrow G\hookleftarrow U_D$ gives 
\begin{equation}\label{eq:comp-and-localization}
\begin{CD}
R\pi_!Rj_{D*}j_D^{*}\cL(d)[2d-1]@>>>\pi_{D!}\pi_D^{*}\Lambda(\sH)@>>>
R\pi_!\cL(d)[2d]\\
@VV\comp^{k}V@VV\comp^{k}V@VV\comp^{k}V\\
R\pi_!Rj_{D*}j_D^{*}\Log{k}(d)[2d-1]@>>>\pi_{D!}\pi_D^{*}\Log{k}@>>>
R\pi_!\Log{k}(d)[2d]
\end{CD}
\end{equation}
compatible with the transition maps $\Log{k}\to \Log{k-1}$. 
\begin{proposition}
There is a commutative diagram of short exact sequences
\[
\xymatrix{
0\ar[r] & H^{2d-1}(S,R\pi_!Rj_{D*}j_D^{*}\cL(d))\ar[r]^{\res}\ar[d]^{\comp}& H^{0}(S,\pi_{D!}\pi_D^{*}\Lambda(\sH))\ar[r]\ar[d]^{e^{*}\comp} & H^{0}(S,\Zp)\ar[r] \ar[d]&0\\
0\ar[r] & H^{2d-1}(S,R\pi_!Rj_{D*}j_D^{*}\cLog_\Qp(d))\ar[r]^{\res}& H^{0}(S,\pi_{D!}\pi_D^{*}\prod_{k\ge 0}\Sym^{k}\sH_\Qp(d))\ar[r] & H^{0}(S,\Qp)\ar[r] & 0.
}
\]
\end{proposition}
\begin{proof} Take the long exact cohomology sequence of the commutative diagram in \eqref{eq:comp-and-localization}, tensor the 
lower horizontal line with $\Qp$ and then pass to the inverse limit over $k$. Using 
the isomorphisms in Corollary \ref{cor:comp-isom} gives the commutative diagram as stated.
\end{proof}

\begin{corollary}\label{cor:pol-and-EI}
Let $\alpha\in \Zp[D]^{0}$, with $D=G[c]$ as before. Then one has
\[
\comp(_\alpha\cpol)={_\alpha}\cpol_\Qp
\]
in $H^{2d-1}(S,R\pi_!Rj_{D*}j_D^{*}\cLog_\Qp(d))$.
In particular, for every $N$-torsion section $t:S\to U_D$ one has 
\[
\comp(_\alpha\EI(t))=t^{*}({_\alpha}\cpol_\Qp).
\]
\end{corollary}
\begin{proof}
Immediate from the definition of $_\alpha\cpol$ and $_\alpha\cpol_\Qp$ and the commutative diagram in the proposition. The second statement follows from the first as $\comp$ is compatible with the evaluation map at $t$.
\end{proof}
\subsection{Interpolation of the $\Qp$-Eisenstein classes}
For our main result, we first have to relate the comparison map 
$\comp^{k}$ with the moment map $\mom^{k}$.
\begin{proposition}\label{prop:comp-and-mom}
The composition
\[
\Lambda(\sH)\xrightarrow{e^{*}(\comp^{k})} e^{*}\Log{k}\xrightarrow{\pr_k}{\Gamma}_k(\sH)
\]
coincides with the moment map $\mom^{k}$.
\end{proposition}
\begin{proof}
By the definitions of $\mom^{k}$ and $\comp^{k}$ it suffices to
prove this statement for $\Lambda_r$-coefficients. Consider 
\[
\comp_r^{k}:\Lambda_r[G_{rk}]\to \Log{k}_{\Lambda_r}
\]
from \eqref{eq:comp-map}. This comes by adjunction from a map
\[
\beta_r:\Lambda_r\to [p^{rk}]^{*}\Log{k}_{\Lambda_r},
\]
on $G_{rk}$
which 
has by definition the property that its pull-back $e_{rk}^{*}(\beta_r)$ coincides with $\one{k}:\Lambda_r\to e^{*}\Log{k}_{\Lambda_r}$.
By Lemma \ref{lemma:e-upper-star} the map $\beta_r$ is uniquely determined by this property. As $\Log{k}_{\Lambda_r}$ is unipotent of 
length $k$, the pull-back $[p^{rk}]^{*}\Log{k}_{\Lambda_r}$ 
is trivial by Lemma \ref{lemma:unipotent-trivialization} and is hence
of the form 
\[
[p^{rk}]^{*}\Log{k}_{\Lambda_r}\isom \pi_{rk}^{*}e^{*}\Log{k}_{\Lambda_r} \isom \pi_{rk}^{*}\prod_{i=0}^{k}\Gamma_i(\sH_r),
\]
where the last isomorphism is obtained by the splitting $\one{k}$.
Thus the map 
\begin{align*}
\Lambda_r\to [p^{rk}]^{*}\Log{k}_{\Lambda_r}\isom \pi_{rk}^{*}\prod_{i=0}^{k}\Gamma_i(\sH_r)&& 1\mapsto \sum_{i=0}^{k}\tau_r^{[i]},
\end{align*}
where $\tau_r^{[i]}$ is the $i$-th divided power of the tautological section from \eqref{eq:tau}, has the property  that 
its pull-back by $e_{rk}^{*}$ coincides with $\one{k}$. It follows
that this map equals $\beta_r$ and by definition of the moment map
in \eqref{eq:mom-k-r} the projection to the $k$-th component coincides also with the moment map.
\end{proof}
Let $t:S\to U_D$ be an $N$-torsion section. 
We need  a compatibility between the composition
\[
\mom^{k}_N:=\mom^{k}\circ [N]_\#:\Lambda(\sH\langle t\rangle)\to\Lambda(\sH\langle t\rangle)\to \Gamma_{k}(\sH)
\]
and the map $\varrho_t$ in the splitting principle \ref{cor:splitting_sheaf} composed with
the projection onto the $k$-th component
\[
\pr_k\circ\varrho_t:t^{*}\Log{k}_\Qp\isom \prod_{i=0}^{k}\Sym^{k}\sH_\Qp\to\Sym^{k}\sH_\Qp.
\]
\begin{proposition}\label{prop:mom-and-varrho}
There is a commutative diagram 
\[
\begin{CD}
H^{2d-1}(S,\Lambda(\sH\langle t\rangle)(d) )@>\mom_N^{k}>>
H^{2d-1}(S,{\Gamma}_{k}(\sH)(d))\\
@Vt^{*}\comp^{k} VV@VVV\\
H^{2d-1}(S,t^{*}\Log{k}_\Qp(d))@>N^{k}\pr_k\circ\varrho_t>> H^{2d-1}(S,{\Sym}^{k}\sH_\Qp(d),
\end{CD}
\]
where $\mom_N^{k}=\mom^{k}\circ [N]_\#$ and 
$\varrho_t=[N]_\#^{-1}\circ [N]_\#$.
\end{proposition}
\begin{proof}
The  commutative diagram
\[
\begin{CD}
H^{2d-1}(S,\Lambda(\sH\langle t \rangle)(d))@>[N]_\#>>H^{2d-1}(S,\Lambda(\sH)(d))\\
@Vt^{*}\comp^{k} VV@VVe^{*}\comp^{k} V\\
H^{2d-1}(S,t^{*}\Log{k}_\Qp(d))@>[N]_\#>\isom>H^{2d-1}(S,e^{*}\Log{k}_\Qp(d))
\end{CD}
\]
coming from functoriality of $\comp^{k}$ and the isomorphisms
\begin{align*}
H^{2d-1}(S,t^{*}\Log{k}(d))\otimes_\Zp\Qp&\isom 
H^{2d-1}(S,t^{*}\Log{k}_\Qp(d))\\
H^{2d-1}(S,e^{*}\Log{k}(d))\otimes_\Zp\Qp&\isom 
H^{2d-1}(S,e^{*}\Log{k}_\Qp(d))
\end{align*}
reduces the proof of the proposition to show the commutativity of the diagram
\[
\begin{CD}
H^{2d-1}(S,\Lambda(\sH)(d) )@>\mom^{k}>>
H^{2d-1}(S,{\Gamma}_{k}(\sH)(d))\\
@Ve^{*}\comp^{k} VV@VVV\\
H^{2d-1}(S,e^{*}\Log{k}_\Qp(d))@>N^{k}\pr_k\circ [N]_\#^{-1}>> H^{2d-1}(S,{\Sym}^{k}\sH_\Qp(d).
\end{CD}
\]
The isogeny $[N]$ acts by $N$-multiplication on $\sH$, hence by multiplication with $N^{k}$ on $\Sym^{k}\sH_\Qp$, which means that
\[
\pr_k\circ [N]_\#^{-1}=[N]_\#^{-1}\circ \pr_k=N^{-k}\pr_k.
\]
Thus it remains to show that the diagram 
\begin{equation*}
\begin{CD}
H^{2d-1}(S,\Lambda(\sH)(d))@>\mom^{k}>> H^{2d-1}(S,\Gamma_k(\sH)(d))\\
@Ve^{*}\comp^{k}VV@VVV\\
H^{2d-1}(S,e^{*}\Log{k}_\Qp(d))@>\pr_k>>H^{2d-1}(S,\Sym^{k}\sH_\Qp(d))
\end{CD}
\end{equation*}
commutes, which follows from  Proposition \ref{prop:comp-and-mom}
and the isomorphism 
\[
H^{2d-1}(S,\Gamma_k(\sH)(d))\otimes_\Zp\Qp\isom H^{2d-1}(S,\Sym^{k}\sH_\Qp(d))
\]
which was obtained in Corollary \ref{cor:comp-isom}.
\end{proof}

Recall from Definition \ref{def:Eisenstein-Iwasawa} the Eisenstein-Iwasawa class
\[
_\alpha\EI(t)_N=[N]_\#(_\alpha\EI(t))\in H^{2d-1}(S, \Lambda(\sH)(d))
\]
and from \ref{def:Qp-Eisenstein-class} the $\Qp$-Eisenstein class
\[
_\alpha\Eis^{k}_{\Qp}(t)\in H^{2d-1}(S,\Sym^{k}\sH_\Qp).
\]
We consider its image under the $k$-th moment map 
\[
\mom^{k}:H^{2d-1}(S,\Lambda(\sH)(d) )\to 
H^{2d-1}(S,{\Gamma}_{k}(\sH)(d)).
\]
The main result of this paper can now be formulated as follows:
\begin{theorem}[Interpolation of $\Qp$-Eisenstein classes]\label{thm:interpolation}
The image of $_\alpha\EI(t)_N$ under the $k$-th moment map is given by
\[
\mom^{k}(_\alpha\EI(t)_N)={N^{k}}{_\alpha}\Eis_\Qp^{k}(t).
\]
\end{theorem}
\begin{proof}
This follows by combining Corollaries \ref{cor:pol-and-EI},  \ref{prop:mom-and-varrho} and the definition
of the $\Qp$-Eisenstein class \ref{def:Qp-Eisenstein-class}.
\end{proof}
\begin{remark}
For comparison with \cite[Theorem 12.4.21]{Kings-Soule} we point out again that the normalization of ${_\alpha}\Eis_\Qp^{k}(t)$ in loc. cit. is \emph{different}. We had there a factor of $-N^{k-1}$ in front of the Eisenstein series.
\end{remark}
\bibliographystyle{amsalpha}
\bibliography{integral-polylog}
\end{document}